\RequirePackage{snapshot}
\documentclass{siamart171218}

\usepackage{microtype}
\usepackage{graphicx}
\usepackage{amsmath,amssymb,mathtools}
\usepackage{braket}
\usepackage{enumerate}
\usepackage{hyperref}
\usepackage[numbers,sort]{natbib}
\numberwithin{equation}{section}
\numberwithin{figure}{section}
\usepackage{tikz}
\usetikzlibrary{arrows} \tikzset{>=stealth', node distance=5cm}
\usepackage[framemethod=tikz]{mdframed}
\mdfsetup{%
  skipbelow=10pt,
  skipabove=8pt,
  linewidth=1.25pt,
  backgroundcolor=gray!10,
  userdefinedwidth=\textwidth,
}
\usepackage[text={5.125in,8.25in},centering]{geometry}

\newsiamremark{example}{Example}
\newsiamremark{remark}{Remark}
\newcommand*{\QEDA}{\hfill\ensuremath{\square}}%

\DeclareMathOperator*{\minimize}{minimize}
\DeclareMathOperator*{\maximize}{maximize}
\DeclareMathOperator*{\argmin}{arg\,min}
\DeclareMathOperator*{\argmax}{arg\,max}
\DeclareMathOperator*{\subjectto}{subject\ to}
\DeclareMathOperator{\epi}{epi}

\newcommand{\penv}{{\normalfont\textsf{p}}_{\alpha,f}}
\DeclareMathOperator{\normal}{\mathcal{N}}
\DeclareMathOperator{\cl}{cl}
\DeclareMathOperator{\interior}{int}

\DeclareMathOperator{\pprox}{{\normalfont\textsf{pprox}}}
\DeclareMathOperator{\rpprox}{{\normalfont\mathfrak{P}}_{\alpha,f}}
\DeclareMathOperator{\prox}{{\normalfont\textsf{prox}}}
\DeclareMathOperator{\proj}{{\normalfont\textsf{proj}}}
\DeclareMathOperator{\cone}{cone}
\DeclareMathOperator{\dist}{{\normalfont\textsf{dist}}}
\DeclarePairedDelimiterX{\ip}[2]{\langle}{\rangle}{#1,#2}
\DeclareMathOperator{\dom}{dom}
\newcommand*{\medcap}{\mathbin{\scalebox{1.5}{\ensuremath{\cap}}}}%
\newcommand{\st}{\subjectto}
\newcommand{\infc}{\mathbin{\square}}
\newcommand{\sumc}{\mathbin{\square}}
\newcommand{\maxc}{\mathbin{\diamond}}
\newcommand{\pconv}{\mathbin{\diamond}}
\newcommand{\lscale}[2]{#1^+\!#2}

\newcommand{\braces}[1]{\left\{{#1}\right\}}

\newcommand{\polar}[1]{{#1}^{\circ}}

\newcommand{\R}{\mathbb{R}}
\newcommand{\Rp}{\R_{+}}

\newcommand{\Rpi}{\Rp\cup\{+\infty\}}

\newcommand{\minv}[1]{\min\braces{#1}}
\newcommand{\maxv}[1]{\max\braces{#1}}

\newcommand{\rec}[1]{{#1}_{\scriptscriptstyle\infty\!}}

\newcommand{\Cscr}{\mathcal{C}}
\newcommand{\Dscr}{\mathcal{D}}
\newcommand{\Kscr}{\mathcal{K}}

\newcommand{\Xscr}{\mathcal{X}}
\newcommand{\Yscr}{\mathcal{Y}}

\newcommand{\Cs}{_{\scriptscriptstyle\Cscr}}
\newcommand{\Ds}{_{\scriptscriptstyle\Dscr}}
\newcommand{\Ks}{_{\scriptscriptstyle\Kscr}}
\newcommand{\Ksp}{_{\scriptscriptstyle\Kscr^\circ}}

\makeatletter
\let\save@mathaccent\mathaccent
\newcommand*\if@single[3]{%
  \setbox0\hbox{${\mathaccent"0362{#1}}^H$}%
  \setbox2\hbox{${\mathaccent"0362{\kern0pt#1}}^H$}%
  \ifdim\ht0=\ht2 #3\else #2\fi
  }
\newcommand*\rel@kern[1]{\kern#1\dimexpr\macc@kerna}
\newcommand*\widebar[1]{\@ifnextchar^{{\wide@bar{#1}{0}}}{\wide@bar{#1}{1}}}
\newcommand*\wide@bar[2]{\if@single{#1}{\wide@bar@{#1}{#2}{1}}{\wide@bar@{#1}{#2}{2}}}
\newcommand*\wide@bar@[3]{%
  \begingroup
  \def\mathaccent##1##2{%
    \let\mathaccent\save@mathaccent
    \if#32 \let\macc@nucleus\first@char \fi
    \setbox\z@\hbox{$\macc@style{\macc@nucleus}_{}$}%
    \setbox\tw@\hbox{$\macc@style{\macc@nucleus}{}_{}$}%
    \dimen@\wd\tw@
    \advance\dimen@-\wd\z@
    \divide\dimen@ 3
    \@tempdima\wd\tw@
    \advance\@tempdima-\scriptspace
    \divide\@tempdima 10
    \advance\dimen@-\@tempdima
    \ifdim\dimen@>\z@ \dimen@0pt\fi
    \rel@kern{0.6}\kern-\dimen@
    \if#31
      \overline{\rel@kern{-0.6}\kern\dimen@\macc@nucleus\rel@kern{0.4}\kern\dimen@}%
      \advance\dimen@0.4\dimexpr\macc@kerna
      \let\final@kern#2%
      \ifdim\dimen@<\z@ \let\final@kern1\fi
      \if\final@kern1 \kern-\dimen@\fi
    \else
      \overline{\rel@kern{-0.6}\kern\dimen@#1}%
    \fi
  }%
  \macc@depth\@ne
  \let\math@bgroup\@empty \let\math@egroup\macc@set@skewchar
  \mathsurround\z@ \frozen@everymath{\mathgroup\macc@group\relax}%
  \macc@set@skewchar\relax
  \let\mathaccentV\macc@nested@a
  \if#31
    \macc@nested@a\relax111{#1}%
  \else
    \def\gobble@till@marker##1\endmarker{}%
    \futurelet\first@char\gobble@till@marker#1\endmarker
    \ifcat\noexpand\first@char A\else
      \def\first@char{}%
    \fi
    \macc@nested@a\relax111{\first@char}%
  \fi
  \endgroup
}
\makeatother

\def\rbar{\skew3\bar r}

\def\xbar{\bar{x}}
\def\xhat{\widehat x}

\def\ybar{\skew3\bar y}

\headers{Polar Convolution}{M. P. Friedlander, I. Mac\^edo, and T. K. Pong}

\title{Polar Convolution\thanks{Date: August 21, 2018; revised February 1, 2019.
\funding{MPF is supported by ONR award N00014-17-1-2009; TKP is supported by Hong Kong Research Grants Council award PolyU253008/15p.}}}

\author{%
  Michael P. Friedlander\thanks{Department of Computer Science and Department of Mathematics, University of British Columbia (\email{mpf@cs.ubc.ca}).}%
  \and Ives Mac\^edo\thanks{Amazon (\email{ives.macedo@gmail.com}); this author contributed to this work prior to joining Amazon.}
\and Ting Kei Pong\thanks{Department of Applied Mathematics, Hong Kong Polytechnic University (\email{tk.pong@polyu.edu.hk})}}

\begin{document}

\maketitle

\begin{abstract}
  The Moreau envelope is one of the key convexity-preserving
  functional operations in convex analysis, and it is central to the
  development and analysis of many approaches for convex
  optimization. This paper develops the theory for an analogous
  convolution operation, called the polar envelope, specialized to
  gauge functions. Many important properties of the Moreau envelope
  and the proximal map are mirrored by the polar envelope and its
  corresponding proximal map. These properties include smoothness of
  the envelope function, uniqueness and continuity of the proximal
  map, which play important roles in duality and in the construction
  of algorithms for gauge optimization. A suite of tools with which to
  build algorithms for this family of optimization problems is thus
  established.
\end{abstract}


\section{Motivation}

Let $f_1$ and $f_2$ be proper convex functions that map a finite-dimensional Euclidean space $\Xscr$ to the extended reals $\R\cup\{+\infty\}$.
The \emph{infimal sum convolution} operation
\begin{equation}
  \label{eq:inf-convolution-2}
  (f_1\sumc f_2)(x) = \inf_{\mathclap{z}}\set{f_1(z) + f_2(x-z)}
\end{equation}
results in a convex function that is a ``mixture'' of $f_1$ and $f_2$.
As is usually the case in convex analysis, the operation is best
understood from the epigraphical viewpoint: if the infimal operation
attains its optimal value at each $x$, then the resulting function
$f_1\sumc f_2$ satisfies the relationship
\[
  \epi( f_1\sumc f_2) = \epi f_1 + \epi f_2.
\]
Sum convolution is thus also known as \emph{epigraphical
  addition} \citep[p.34]{roc70}.

This operation emerges in various forms in convex optimization. Take, for example, the function $f_2=\tfrac1{2\alpha}\|\cdot\|^2_2$ for some positive scalar $\alpha$. Then
\begin{equation} \label{eq:7}
  \left( f_1 \infc \tfrac1{2\alpha} \|\cdot\|_2^2\right)(x)
  = \inf_z\left\{f_1(z) + \tfrac1{2\alpha}\|x-z\|_2^2\right\}
\end{equation}
is the Moreau envelope of $f_1$, and the minimizing set
\[
  \prox_{\alpha f_1}(x) :=
  \argmin_z\left\{
    f_1(z) + \tfrac1{2\alpha}\|x-z\|^2_2
  \right\}
\]
is the proximal map of $\alpha f_1$. The proximal map features
prominently in first-order methods for minimizing nonsmooth convex
functions, where it appears as a key computational kernel. For
example, the proximal map is required at each iteration of a splitting
method, including the proximal-gradient and Douglas-Rachford
algorithms, often used in practice. \citet[Ch.~28]{Bauschke2011}
describe this class of methods in detail.

Sum convolution also appears naturally in the process of dualization
through conjugacy. The close relationship between sum convolution and
conjugacy is encapsulated by the formula
\begin{equation} \label{eq:4}
  (f_1\sumc f_2)^* = f_1^* + f_2^*,
\end{equation}
which reveals the duality between sum convolution and
addition~\cite[Theorem 16.4]{roc70}. This fact is often leveraged as a
device to solve nonsmooth problems via algorithms that are restricted
to differentiable functions, as follows: the objective function is
regularized, which results, via conjugacy, in a dual objective that is
the sum convolution of the conjugated regularized objective. Under the
appropriate conditions, Fenchel duality provides the needed
correspondence, as illustrated in \cref{fig:fenchel}.  The key
property of the regularized dual objective
$(f_1^*\sumc \tfrac1{2\alpha}\|\cdot\|_2^2)$, shown in the bottom
right of this figure, is that its gradient is
$\tfrac1\alpha$-Lipschitz
continuous~\cite[Prop.~12.29]{Bauschke2011}. A variety of first-order
algorithms can then be applied to the regularized dual problem in
order to approximate a solution of the regularized primal. This
approach is often used in practice. For example, it forms the backbone
for the TFOCS software package \citep{beckbobcandgrant:2011} and for
Nesterov's smoothing algorithm and its variants
\citep{beckteboulle:2012,nesterov:2005}.

\begin{figure}[t]\centering
\includegraphics[page=1]{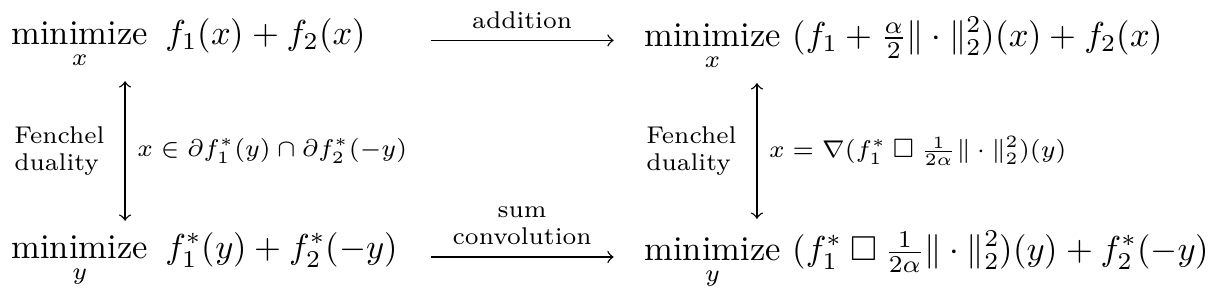}
\caption{The Fenchel duality correspondence for regularized
  problems. The formulas on the vertical arrows show how a primal
  solution $x$ may be recovered from a dual solution $y$
  \citep[Example~11.41]{rtrw:1998}. Strong convexity of the
  regularized problem ensures that $x$ can be recovered
  uniquely.\label{fig:fenchel}}
\end{figure}

The approach illustrated in \cref{fig:fenchel} is fully general,
in the sense that it applies to all of convex
optimization. %
However, such generality may unravel useful structures already present
in the original problem. Homogeneity, in particular, is not preserved
under the Moreau envelope.
This property
often appears in sparse optimization problems, and more generally, in
gauge optimization problems, which involve nonnegative sublinear
functions \citep{FriedlanderMacedoPong:2014}. The smoothing approach
described in \cref{fig:fenchel} does not apply in the gauge setting:
even if the problems on the left-hand side of the figure are gauge
problems, the functions on the right-hand side cannot be gauges,
because the regularization is not homogeneous. We are thus led to
consider an alternative convexity-preserving convolution operation
that can be used to construct smooth approximations and that appears
naturally as part of the dualization process.

\emph{Infimal max convolution} is a convexity-preserving
transformation similar to infimal sum convolution. However, the sum
operation in~\eqref{eq:inf-convolution-2} is replaced by a ``max''
between the two functions:
\begin{equation*}
  (f_1\maxc f_2)(x) =
  \inf_{z}\max\Set{f_1(z),\,f_2(x-z)}.
\end{equation*}
Max convolution first appears in \citet[Theorem 5.8]{roc70} as an
example of a convexity-preserving operation, and subsequently is studied
in detail by \citet{SeegerVolle95}.  As with sum convolution, this
operation mixes $f_1$ and $f_2$, except that it results in a function
whose level sets are the sums of the level sets of $f_1$ and
$f_2$. Seeger and Volle thus refer to max convolution as
\emph{level-set addition}.

Our main purpose in this paper is to develop further the theory of the
max-convolution operation, particularly when applied to gauges. As we
will describe, this operation, coupled with gauge duality, leads to a
correspondence that is completely analogous to the transformations
shown in \cref{fig:fenchel}.

\subsection{Contributions}

The duality between conjugacy and sum convolution, encapsulated
in~\eqref{eq:4}, is one of the main ingredients necessary for deriving
the smooth dual approach described by
\cref{fig:fenchel}. %
We derive in \cref{sec:polar-convolution} an analogous identity for
max convolution for gauges. When $\kappa_1$ and $\kappa_2$ are gauges,
\[
  (\kappa_1 \maxc \kappa_2)^\circ = \kappa_1^\circ + \kappa_2^\circ.
\]
Here, the map $\kappa\mapsto \kappa^\circ$ is the polar transform of
the gauge $\kappa$. This identity shows that for gauge functions,
polarity enjoys a relationship with max convolution that is analogous to
the relationship between Fenchel conjugation and sum convolution.
Because many useful properties related to polarity accrue when
specializing max convolution to gauge functions, we term the operation
\emph{polar convolution}.

A case of special importance is the polar convolution of a gauge with
the 2-norm, rather than the squared 2-norm used
in~\eqref{eq:7}. Under mild conditions,
\[
  \left(\kappa \maxc \tfrac1\alpha\|\cdot\|_2 \right)(x)
  =\inf_z\max\left\{ \kappa(z),\, \tfrac1\alpha\|x-z\|_2 \right\}
\]
is smooth at $x$, as we prove in \cref{sec:polar-envelope}. We
dub this function the \emph{polar envelope} of $\kappa$. In that
section, we describe the subdifferential properties of the polar
envelope, along with properties of the corresponding polar proximal map. In
\cref{sec:computing_polar_env} we demonstrate how to compute these
quantities for several important examples.

Consider again the duality correspondence shown in
\cref{fig:fenchel}. We derive in \cref{sec:smooth_dual} an analogous
correspondence that replaces Fenchel duality with gauge duality, and
sum convolution with max convolution. These relationships are
summarized by \cref{fig:gauge}. In this figure, $\kappa$ is a closed
gauge, $\Cscr$ is a closed convex set that does not contain the
origin, and $\Cscr':= \set{y\mid \ip{x}{y} \ge 1}$ is the antipolar
set of $\Cscr$. As is the case with Fenchel duality, the regularized
dual problem is now smooth (cf.~\cref{sec:polar-envelope}), and thus
first-order methods may be applied to this problem. The inclusion on
the left-hand side of \cref{fig:gauge} follows from
\citet[Corollary~3.7]{ABDFM18}.

\begin{figure}[t]\centering
\includegraphics[page=2]{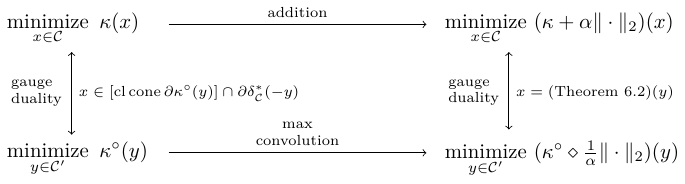}
\caption{The gauge duality correspondence for regularized gauge
  problems. The formulas on the vertical arrows show
  how a primal solution $x$ may be recovered from a dual solution
  $y$. Under gauge duality, the homogeneous regularizer ensures that $x$ can be recovered uniquely.
  \label{fig:gauge}}
\end{figure}

As another illustrative example of how the polar envelope might be
applied, we consider in \cref{sec:polar-prox-algo} how the polar
envelope can be used to develop proximal-point-like algorithms for
convex optimization problems.

\subsection{Notation}

Throughout this paper, the 2-norm of a vector $x$ is denoted by
$\|x\|_2=\sqrt{\ip x x}$, but we often simply denote it by $\|x\|$.
The $\alpha$-level set for a function $f$ is denoted by
$[f\le\alpha]=\set{x|f(x)\le\alpha}$. For a convex set $\Cscr$, let
$\dist\Cs(x):=\inf\set{\|x-z\| | z\in\Cscr}$ denote the Euclidean
distance of $x$ to the set $\Cscr$, and
$\proj\Cs(x):=\argmin\set{\|x-z\||z\in\Cscr}$ denote the corresponding
projection. The recession cone of $\Cscr$ is denoted by
$\Cscr_\infty$, its polar is denoted by
$\Cscr^\circ := \set{u\mid \ip{u}{x} \le 1\ \forall x\in \Cscr}$. When
$\Cscr$ is a cone, its polar simplifies to the closed cone
$\Cscr^\circ=\set{u\mid \ip{u}{x} \le 0\ \forall x\in
  \Cscr}$. Finally, fractions such as $(1/(2\alpha))$ are often
abbreviated as $(1/2\alpha)$.

\section{The max convolution operation}

In this section we review the max convolution operation, and establish
a baseline set of results for later sections. Max convolution is
also known as a \emph{level sum} \citep{SeegerVolle95} because
$(f_1\maxc f_2)$ is the sum of the strict level sets of $f_1$ and
$f_2$:
\begin{equation}\label{eq:6}
  [(f_1\maxc f_2)<\lambda] = [f_1<\lambda] + [f_2<\lambda]\quad\forall\lambda\in\R;
\end{equation}
see \citet[Page~40]{roc70}. This property mirrors that for
sum convolution, where instead it is the strict epigraphs that are
summed:
\[
  \set{(x,\alpha)|(f_1\sumc f_2)(x)<\alpha}
  =\set{(x_1,\alpha_1)|f_1(x_1)<\alpha_1}
  +\set{(x_2,\alpha_2)|f_2(x_2)<\alpha_2}.
\]
In both cases, if the infimal operations that define these
convolutions are attained, then the convolution is \emph{exact}, and
all of the strict inequalities become weak. \citet{Seeger11} uses the
term \emph{inverse sum} to describe the convolution $(f_1\maxc f_2)$
when $f_1$ and $f_2$ are continuous nonnegative sublinear
functions. We follow \citet[Theorem~2.1.3(ix)]{Za02} and adopt the
term ``max convolution" to highlight its convolutional nature.

The perspective transform for any proper convex function
$f:\Xscr\to\R\cup\{+\infty\}$ creates a convex function whose epigraph
is $\cone(\epi f\times\{1\})$, which is the cone generated by the
``lifted'' set. The perspective transform does not necessarily
preserve closure, and so it is convenient for us to redefine a closed
version by
\begin{equation} \label{eq:12}
f^\pi(x,\lambda) := \begin{cases}
  \lambda f\left(\lambda^{-1}x\right) & {\rm if}\ \lambda > 0,\\
  \rec{f}(x) & {\rm if}\ \lambda = 0,\\
  \infty & {\rm if}\ \lambda < 0,
\end{cases}
\end{equation}
where $\rec f$ is the recession function of $f$
\citep[Definition~2.5.1]{AuT03}.  The perspective $f^\pi$ of a proper
closed convex function $f$ is proper closed convex~\cite[Page~67]{roc70}.

\citet{SeegerVolle95} derive an expression for the conjugate of the
max convolution. (Seeger and Volle attribute this result to
\citet{Attouch:90} and \citet{Zalinescu:inprep}.) Here we provide a
slightly rephrased version of that result expressed in terms of the
perspective function. We adopt the notation
\[
  (\lscale{\lambda}{f})(x) :=
  \begin{cases}
    \lambda f(x) & \mbox{if $\lambda > 0$,}
  \\\delta_{\dom f}(x) & \mbox{if $\lambda=0$,}
  \end{cases}
\]
which emphasizes the role of the limiting behavior as $\lambda\to0$ from the right.
Thus,
\[
  \partial(\lscale{\lambda}f)(x) =
  \begin{cases}
    \set{\lambda z \mid z\in\partial f(x)} & \mbox{if $\lambda>0$,}
  \\\normal(x\mid\dom f) & \mbox{if $\lambda=0$.}
  \end{cases}
\]

\begin{proposition}[Seeger and Volle {[1995; Proposition~4.1]}]
  Let $f_1$ and $f_2$ be proper convex functions. Then
  \begin{equation*}%
    (f_1\maxc f_2)^*(y){}
    = \inf \set{ (f_1^*)^\pi(y,\lambda_1)
               + (f^*_2)^\pi(y,\lambda_2) | \lambda_1+\lambda_2=1,\ \lambda_i\ge0}.
  \end{equation*}
\end{proposition}
\begin{proof}
  \citet[Proposition~4.1]{SeegerVolle95} show that
  \begin{equation}\label{conjugate_maxc}
    (f_1\maxc f_2)^*(y)
    = \inf \set{([\lscale{\lambda_1}{f_1}]^* + [\lscale{\lambda_2}{f_2}]^*)(y) |
      \lambda_1+\lambda_2=1,\ \lambda_i\ge0}.
  \end{equation}
  For a proper convex function $g$ and any nonnegative scalar
  $\lambda$,
  \[
    (\lscale{\lambda}{g})^*(y) = \begin{cases}
      \lambda g^*(\lambda^{-1}y) & {\rm if}\ \lambda > 0,\\
      \delta_{\dom g}^*(y) & {\rm if}\ \lambda = 0.
    \end{cases}
  \]
  The desired conclusion thus follows from~\eqref{conjugate_maxc},
  from the definition of perspective functions, and from
  \citet[Theorem~2.5.4(a)]{AuT03}, which asserts
  $\delta^*_{\dom g} = \rec{(g^*)}$.
\end{proof}

We also recall the following formula for the subdifferential of the
max convolution $f_1\maxc f_2$, originally given by
\citet[Proposition~4.3]{SeegerVolle95}. This result will be useful in
establishing the differential properties of the max convolution
applied to gauge functions; cf. \cref{sec:polar-envelope} and
\cref{sec:polar-prox-algo}.

\begin{proposition}[Seeger and Volle {[1995]}]\label{Prop_sub}
  Let $f_1$ and $f_2$ be proper closed convex functions. Let
  $\xbar\in \dom (f_1\maxc f_2)$ and suppose that
  $(f_1\maxc f_2)(\xbar) = \max\set{f_1(\xbar_1),\,f_2(\xbar_2)}$ for
  some $\xbar_1+\xbar_2 = \xbar$, with $\xbar_1\in \dom f_1$ and
  $\xbar_2\in \dom
  f_2$. %
  Then
  \[
    \partial (f_1\maxc f_2)(\xbar) =
    \bigcup \Set{%
      \partial(\lscale{\mu}{f_1})(\xbar_1)
      \cap
      \partial(\lscale{[1-\mu]}{f_2})(\xbar_2) |
      \mu\in\Gamma},
  \]
  where $\Gamma=\set{\mu\in [0,1]| \lscale{\mu}{f_1}(\xbar_1) + \lscale{(1-\mu)}{f_2}(\xbar_2) = (f_1\maxc f_2)(\xbar)}$.
\end{proposition}

\section{Polar convolution}\label{sec:polar-convolution}

Any nonnegative convex function that is positively homogeneous and
zero at the origin is called a {\em gauge}. This family of functions
includes, for example, all norms and semi-norms, and any support
function on a set whose convex hull includes the origin.  When max
convolution is specialized to gauge functions, the analogues between
the max convolution and the sum convolution deepen. Here is our formal
definition of polar convolution.
\begin{definition}[Polar convolution]
  Let $\kappa_1$ and $\kappa_2$ be gauges. Their
  \emph{polar convolution} is defined by $\kappa_1\pconv\kappa_2$.
\end{definition}
Although the max convolution of two general proper convex functions can
be improper (i.e., it may attain $-\infty$), the polar convolution of
two gauges is necessarily a proper convex function. Indeed, one can
show that $\kappa_1\pconv\kappa_2$ is a gauge. Moreover, we make the
following immediate observation, whose simple proof is omitted.
\begin{proposition}[Lower approximation]
  \label{cor:bounded} Let $\kappa_1$ and $\kappa_2$ be gauges. Then
  $\kappa_1\pconv\kappa_2$ is a gauge and
  $(\kappa_1\pconv\kappa_2)(x)\leq\minv{\kappa_1(x),\,\kappa_2(x)}$
  for all $x\in\Xscr$. In particular, if either $\dom\kappa_1=\Xscr$ or
  $\dom\kappa_2=\Xscr$, then $\dom(\kappa_1\pconv\kappa_2)=\Xscr$.
\end{proposition}

In the same way that every convex function can be paired with its Fenchel conjugate, every gauge function $\kappa$ can be paired with its polar, defined by
\begin{equation} \label{eq:5}
\kappa^\circ(y) = \inf\set{\mu > 0\mid \ip{x}{y} \le \mu\kappa(x) \ \forall x}.
\end{equation}
This leads to the polar-gauge inequality
\begin{equation} \label{eq:9}
  \ip x y \le \kappa(x)\cdot\kappa^\circ(y) \quad
  \forall x\in\dom\kappa,\ \forall y\in\dom\kappa^\circ,
\end{equation}
and thus the polar $\polar\kappa$ is the function that satisfies this
inequality most tightly. The following lemma reveals the duality
between polar convolution and addition under the polarity operation.
\begin{lemma}[Polar convolution identity]
  \label{lem1}
  Let $\kappa_1$ and $\kappa_2$ be gauges. Then
  \begin{equation}\label{haharel1}
    (\kappa_1\pconv\kappa_2)^\circ=\kappa_1^\circ+\kappa_2^\circ.
  \end{equation}
  If either $\kappa_1$ or $\kappa_2$ is also continuous, then
  \begin{equation*}
    \kappa_1\pconv\kappa_2=(\kappa_1^\circ+\kappa_2^\circ)^\circ.
  \end{equation*}
\end{lemma}
\begin{proof}
  It follows directly from the definition~\eqref{eq:5} of a polar gauge $\kappa$ that
  \[
    \polar{\kappa}(y) = \sup_x\set{\ip{x}{y}\mid\kappa(x)\le 1}
  \]
  for any $y$. Thus, a direct computation shows that for any $y$,
  \begin{equation*}
    \begin{split}
      \polar{(\kappa_1\pconv\kappa_2)}(y)
        & = \sup_x\set{\ip{x}{y}\mid(\kappa_1\pconv\kappa_2)(x)\le 1} \overset{\rm(i)}= \sup_x\set{\ip{x}{y}\mid(\kappa_1\pconv\kappa_2)(x)< 1}\\
        & = \sup_{x,\,x_1}\set{\ip{x}{y}\mid\kappa_1(x_1)< 1,\, \kappa_2(x-x_1) < 1}\\
        & = \sup_{x,\,x_1}\set{\ip{x-x_1 + x_1}{y}\mid\kappa_1(x_1)< 1,\, \kappa_2(x-x_1) < 1}\\
        & \overset{\rm(ii)}= \sup_{x_1,\,x_2}\set{\ip{x_1 + x_2}{y}\mid\kappa_1(x_1)\le 1,\, \kappa_2(x_2) \le 1}\\
        & = \kappa_1^\circ(y) + \kappa_2^\circ(y),
    \end{split}
  \end{equation*}
  where equalities (i) and (ii) follow from \citet[Theorem~7.6]{roc70}
  and the continuity of the linear function $x\mapsto \ip{x}{y}$. Now,
  if in addition either $\kappa_1$ or $\kappa_2$ is continuous, then
  we see from \cref{cor:bounded} that
  $\kappa_1\pconv\kappa_2$ is continuous. The second conclusion now
  follows immediately by taking the polar of both sides
  of~\eqref{haharel1} and applying
  \citet[Proposition~2.1(ii)]{FriedlanderMacedoPong:2014} to the
  continuous gauge $\kappa_1\pconv\kappa_2$.
\end{proof}

\subsection{Polar convolution as a sum of sets}

Polar convolution can also be viewed as a function induced by a convex
set that defines its unit level set. This connection is most
transparent when viewing gauges via their representation as a
Minkowski functional $\gamma\Ds$ for some nonempty convex set
$\Dscr$:
\[
  \kappa(x) = \gamma\Ds(x) := \inf\set{\lambda\ge0|x\in\lambda\Dscr}.
\]
We first need the following result, which relates the sum of Minkowski
functions of two sets to the Minkowski function of a sum of the
sets.
\begin{lemma} \label{lem:sum-gauges}
  Let $\Dscr_1$ and $\Dscr_2$ be closed convex sets containing the origin. Then
  \begin{equation}\label{eq:2}
    \gamma_{\scriptscriptstyle\Dscr_1} + \gamma_{\scriptscriptstyle\Dscr_2} = \gamma_{\scriptscriptstyle(\Dscr_1^\circ+\Dscr_2^\circ)^\circ}.
  \end{equation}
  If additionally $0\in{\rm ri}(\Dscr_1 - \Dscr_2)$, then
  $(\gamma_{\scriptscriptstyle\Dscr_1} + \gamma_{\scriptscriptstyle\Dscr_2})^\circ =
  \gamma_{\scriptscriptstyle\Dscr_1^\circ+\Dscr_2^\circ}$.
\end{lemma}
\begin{proof}
  Theorem~14.5 of~\cite{roc70} contains most of the needed tools. In
  particular, the gauge of any closed convex set containing the
  origin is the support function of the polar. Thus,
  \[
    \gamma_{\scriptscriptstyle\Dscr_1} + \gamma_{\scriptscriptstyle\Dscr_2}
    = \delta^*_{\scriptscriptstyle\Dscr_1^\circ}
    + \delta^*_{\scriptscriptstyle\Dscr_2^\circ}.
  \]
  Next, observe that
  \[
    \delta^*_{\scriptscriptstyle\Dscr_1^\circ} + \delta^*_{\scriptscriptstyle\Dscr_2^\circ}
    = \delta^*_{\scriptscriptstyle\Dscr_1^\circ+\Dscr_2^\circ}
    = \delta^*_{\scriptscriptstyle\cl(\Dscr_1^\circ+\Dscr_2^\circ)}
    = \delta^*_{\scriptscriptstyle(\Dscr_1^\circ+\Dscr_2^\circ)^{\circ\circ}},
  \]
  where the second equality holds because the support function does
  not distinguish a set from its closure, and the last equality
  follows from the bipolar theorem~\cite[Theorem~14.5]{roc70}. Again
  using the polarity correspondence between gauge and support
  functions,
  $\gamma_{\scriptscriptstyle\Dscr_1} + \gamma_{\scriptscriptstyle\Dscr_2} =
  \gamma_{\scriptscriptstyle(\Dscr_1^\circ+\Dscr_2^\circ)^\circ}$, as required.

  We now prove the second part of the lemma. From~\eqref{eq:2} and
  \citet[Theorem~15.1]{roc70}, we deduce that
  \begin{equation}\label{eq:11}
    (\gamma_{\scriptscriptstyle\Dscr_1}+\gamma_{\scriptscriptstyle\Dscr_2})^\circ
    = \gamma_{\scriptscriptstyle(\Dscr_1^\circ+\Dscr_2^\circ)^{\circ\circ}}.
  \end{equation}
  Next, note that for any closed convex set $\Dscr$ that contains the origin, one has
  \begin{equation}\label{hehehaha1}
    \dom \delta^*_{\Dscr^\circ} = \bigcup_{\lambda > 0}\set{x\mid \delta^*_{\Dscr^\circ}(x)\le \lambda} = \bigcup_{\lambda > 0}\lambda\Dscr^{\circ\circ} = \bigcup_{\lambda > 0}\lambda\Dscr,
  \end{equation}
  where the last equality follows the bipolar
  theorem~\cite[Theorem~14.5]{roc70}. Take the relative interior on
  both sides of \eqref{hehehaha1} and use \citep[p.~50]{roc70} to
  deduce that
  ${\rm ri}\Dscr\subseteq {\rm ri}\dom \delta^*_{\Dscr^\circ}$.  Thus,
  from the assumption $0\in{\rm ri}(\Dscr_1 - \Dscr_2)$, we obtain
  that
  \begin{equation}\label{hehehaha2}
    0 \in {\rm ri}(\Dscr_1 - \Dscr_2) = {\rm ri}\Dscr_1 - {\rm ri}\Dscr_2\subseteq {\rm ri}(\dom \delta^*_{\Dscr_1^\circ}-\dom \delta^*_{\Dscr_2^\circ}),
  \end{equation}
  where the equality follows from \cite[Corollary~6.6.2]{roc70}. Now, the relation \eqref{hehehaha2} together with \cite[Theorem~23.8]{roc70} implies that
  \[
  \Dscr_1^\circ + \Dscr_2^\circ = \partial \delta^*_{\Dscr_1^\circ}(0) + \partial \delta^*_{\Dscr_2^\circ}(0) = \partial(\delta^*_{\Dscr_1^\circ} + \delta^*_{\Dscr_2^\circ})(0),
  \]
  where we used the fact that $\Dscr = \partial \delta^*_{\Dscr}(0)$
  for any closed convex set $\Dscr$ \citep[Theorem~13.1]{roc70}. This
  proves that $\Dscr_1^\circ+\Dscr_2^\circ$ is closed. Thus,
  $(\Dscr_1^\circ+\Dscr_2^\circ)^{\circ\circ}=\Dscr_1^\circ+\Dscr_2^\circ$
  by the bipolar theorem~\cite[Theorem~14.5]{roc70}. It then follows
  from~\eqref{eq:11} that
  $(\gamma_{\Dscr_1}+\gamma_{\Dscr_2})^\circ=\gamma_{\scriptscriptstyle(\Dscr_1^\circ+\Dscr_2^\circ)}$,
  as required.
\end{proof}

When $\Dscr_1$ and $\Dscr_2$ are closed convex sets containing the
origin with $0\in \interior{\Dscr_1^\circ}$---which implies that
$\Dscr_1$ is bounded---\cref{lem:sum-gauges} allows us to express
the polar convolution of two gauge functions
$\gamma_{\scriptscriptstyle\Dscr_1}$ and
$\gamma_{\scriptscriptstyle\Dscr_2}$ as
\[
  \cl{(\gamma_{\scriptscriptstyle\Dscr_1} \pconv \gamma_{\scriptscriptstyle\Dscr_2})}
  = (\gamma_{\scriptscriptstyle\Dscr_1^\circ}+\gamma_{\scriptscriptstyle\Dscr_2^\circ})^\circ
  = \gamma_{\scriptscriptstyle\Dscr_1 + \Dscr_2},
\]
where the first equality follows from \cref{lem1} and
\citet[Propositions~2.1(ii) and~2.3(i)]{FriedlanderMacedoPong:2014},
and the second equality follows from \cref{lem:sum-gauges} and
the bipolar theorem \cite[Theorem~14.5]{roc70}. Thus, we observe that
the polar convolution of the Minkowski functions of two sets results
in the Minkowski function of the sum of the sets. This result
confirms the level-set addition property described by the
identity~\eqref{eq:6}.

\section{Polar envelope} \label{sec:polar-envelope}

We now define a special case of polar convolution in which one of the
functions involved is a multiple of the 2-norm. This operation is
analogous to the \textit{Moreau envelope} of a general convex function
$f$; cf.~\eqref{eq:7}.  \citet[Example~3.1]{SeegerVolle95} briefly
describe the max convolution of a general convex function with
the (unsquared) 2-norm. When restricted to gauge functions, the max
convolution operation has a number of useful properties that can be
characterized explicitly. These include, for example, its differential
properties. Here is the formal definition of the polar envelope of a
gauge function.
\begin{definition}[Polar envelope]
  For a gauge $\kappa$ and positive scalar $\alpha$, the function
  \begin{equation} \label{eq:10}
    \kappa_\alpha(x)
    := \left(\kappa\pconv(1/\alpha)\|\cdot\|\right)(x)
    =\inf_z\max\left\{ \kappa(z),\, (1/\alpha)\|x-z\| \right\}
  \end{equation}
  is the \emph{polar envelope} of $\kappa$.  The corresponding
  \emph{polar proximal map}
  \begin{equation*}
    \pprox_{\alpha \kappa}(x) := \argmin_z\,\maxv{\kappa(z),\,(1/\alpha)\|x-z\|}
  \end{equation*}
  is the minimizing set that defines $\kappa_\alpha$.
\end{definition}

\cref{fig:envelopes} shows the Moreau and polar envelopes of the infinity norm.

\begin{figure}[t]
  \centering
  \includegraphics[width=\textwidth,trim=0 70 0 70,clip]{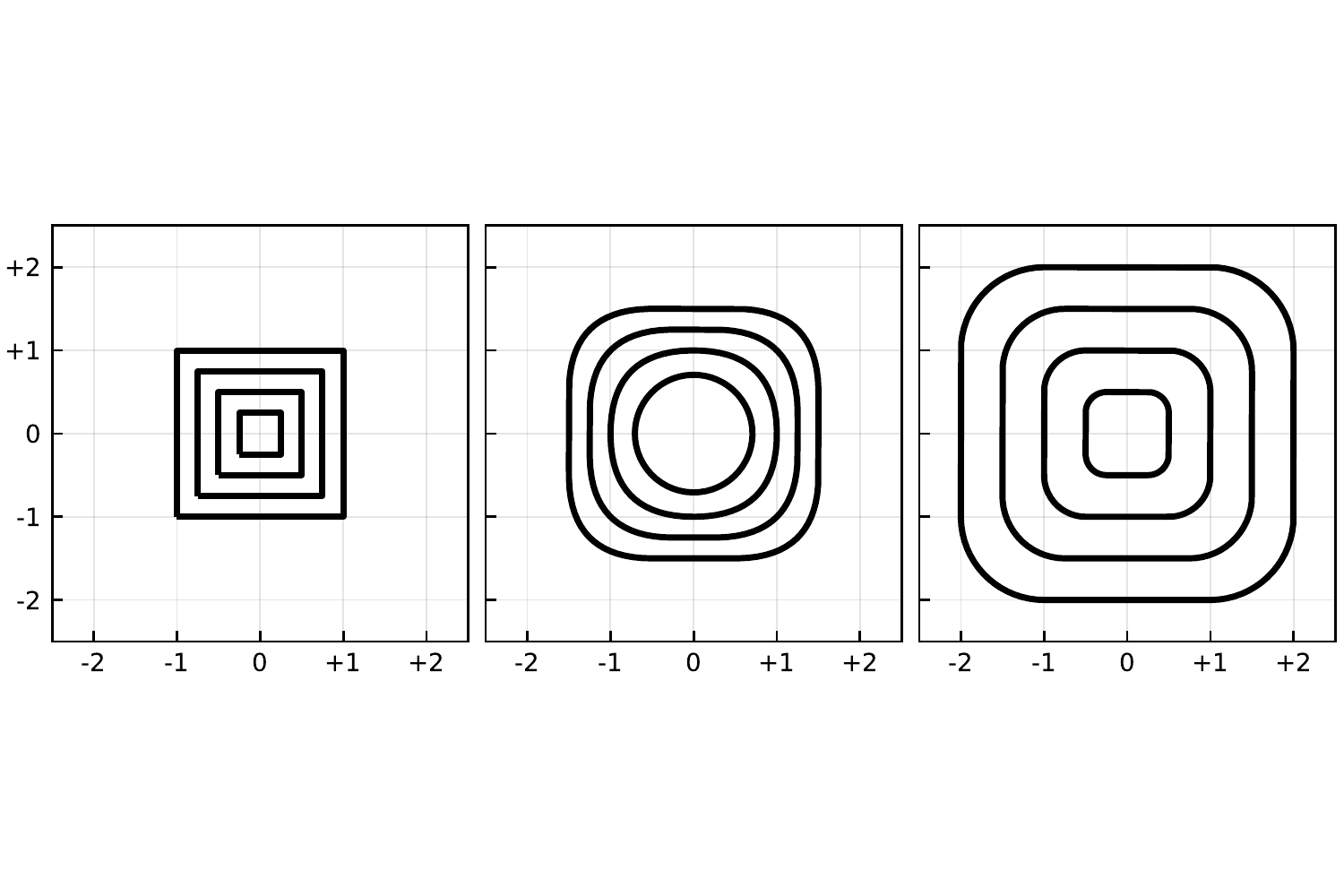}
  \caption{The $\{.25,.50,.75,1.0\}$ iso-contours of the infinity norm
    (left), its Moreau envelope (middle), and its polar envelope
    (right). The iso-contours of the original norm and its polar
    envelope are evenly spaced because both functions are positively
    homogeneous, unlike the Moreau envelope. }
  \label{fig:envelopes}
\end{figure}

\begin{example}[Indicator to a cone]
  \label{ex:indicator-to-cone}
  Let $\Kscr$ be a closed convex cone. Then $\kappa := \delta\Ks$ is a
  closed gauge. Moreover, for any $\alpha > 0$, we have
  \[
    \kappa_\alpha(x)
    = \inf_z \maxv{\delta_\Kscr(z),\,(1/\alpha)\|x-z\|}
    = \inf_{z\in\Kscr}(1/\alpha)\|x-z\|
    = (1/\alpha)\dist\Ks(x).
  \]
  For comparison, the Moreau envelope of $\delta\Ks$ is related to the
  squared distance to the cone:
  \[
    \big( \delta\Ks \sumc (1/2\alpha) \|\cdot\|^2\big)(x)
    = (1/2\alpha)\dist\Ks^2(x).
  \]
  The solutions of both the max and sum convolutions are the same
  in this case, and correspond to the projection of $x$ onto the cone:
  \[
    \pprox_{\alpha\kappa}(x) = \proj\Ks(x) = \prox_{\alpha\kappa}(x).
  \]
  \QEDA
\end{example}

We collect some known properties of the polar envelope in the next
proposition, which specializes results established by
Proposition~3.2 and Corollary~4.1 of \citet{SeegerVolle95}.
\begin{proposition}[Seeger and Volle {[1995]}]\label{known}
  Let $\kappa$ and $\rho$ be gauge functions. The following properties
  hold.
  \begin{enumerate}[{\rm (i)}]
  \item for any $\alpha > 0$, $\kappa_\alpha$ is Lipschitz continuous
    with modulus $1/\alpha$;
  \item $\cl \kappa(x) = \sup\set{\kappa_\alpha(x)|\alpha > 0}$ for
    all $x$;
  \item if $\kappa$ and $\rho$ are both closed, then $\kappa = \rho$
    if and only if there exists $\alpha > 0$ so that
    $\kappa_\alpha = \rho_\alpha$.
  \end{enumerate}
\end{proposition}

The following results establish the differential properties of the
polar envelope and the corresponding polar proximal map.

\begin{theorem}[Differential properties]
  \label{th:polar-envelope-subdiff}
  For any gauge $\kappa$ and positive scalar $\alpha$, the following
  properties hold.
 \begin{enumerate}[{\rm (i)}]
 \item The subdifferential of the polar envelope $\kappa_\alpha$ is given by
   \[
     \partial \kappa_\alpha(x) =
     \argmax_y\left\{ \ip{x}{y}\mid \kappa^\circ(y) + \alpha\|y\|\le 1\right\}.
   \]
   Moreover, $\kappa_\alpha(x) = \ip{x}{y}$ for any $y\in \partial \kappa_\alpha(x)$.
 \item If $\xbar\in\pprox_{\alpha\kappa}(x)$, then
   $\|x - \xbar\| \ge \alpha\kappa(\xbar)$. If in addition $\kappa$
   is continuous, then $\|x - \xbar\| = \alpha\kappa(\xbar)$.
 \item If $\kappa$ is closed, then $\pprox_{\alpha\kappa}(x)$ is a
   singleton for all $x$, and $\pprox_{\alpha\kappa}$ is a
   continuous and positively homogeneous map.
 \item Suppose that $\kappa$ is closed. Then $\kappa_\alpha$ is
   differentiable at all $x$ such that $\kappa_\alpha(x)>0$. Moreover,
   at these $x$, it holds that $\ip{x}{x-\xbar}>0$ and
   \[
   \nabla \kappa_\alpha(x) = \frac{\|x-\xbar\|}{\alpha \ip{x}{x-\xbar}}(x-\xbar),
   \]
   where $\xbar = \pprox_{\alpha\kappa}(x)$.
 \end{enumerate}
\end{theorem}
\begin{proof}
  Part (i). Because $\kappa_\alpha$ is continuous, we have from
  \citet[Proposition 2.1{(ii--iii)}]{FriedlanderMacedoPong:2014} that
 \begin{equation}\label{hahaha}
 \begin{aligned}
   \kappa_\alpha(x)
   &= (\kappa_\alpha)^{\circ\circ}(x)
   = \sup_x\set{\ip{x}{y}\mid (\kappa_\alpha)^\circ(y)\le 1}\\
   &= \sup_x\set{\ip{x}{y}\mid \kappa^\circ(y) + \alpha\|y\|\le 1},
 \end{aligned}
 \end{equation}
 where the last equality follows from \cref{lem1}. Since the
 set $\set{y\mid \kappa^\circ(y) + \alpha\|y\|\le 1}$ is compact, the
 subdifferential formula follows immediately from
 \citet[Theorem~4.4.2, p.189]{hiriart-urruty01}. The claim concerning
 $\kappa_\alpha(x)$ now follows from this and~\eqref{hahaha}.

 Part (ii). Suppose to the contrary that
 $\|x - \xbar\| < \alpha\kappa(\xbar)$. Then
 $\kappa(\xbar) = \kappa_\alpha(x)$ and there exists $t\in (0,1)$ so
 that
  \[
  (1/\alpha)\|x - t\xbar\| < \kappa(t\xbar) < \kappa(\xbar) = \kappa_\alpha(x).
  \]
  Hence,
  $\kappa_\alpha(x) > \max\set{\kappa(t\xbar),\,(1/\alpha)\|x - t\xbar\|} \ge \kappa_\alpha(x)$, which leads to a contradiction. Now
  suppose in addition that $\kappa$ is continuous and suppose to the
  contrary that $\|x - \xbar\| > \alpha\kappa(\xbar)$. Then
  $(1/\alpha)\|x - \xbar\| = \kappa_\alpha(x)$ and there exists
  $t\in (0,1)$ so that
 \[
   \kappa_\alpha(x)
   = \alpha^{-1} \|x - \xbar\|
   > \alpha^{-1}(1-t)\|x - \xbar\|
   = \alpha^{-1}\|x - (\xbar + t[x - \xbar])\|
   > \kappa(\xbar + t(x - \xbar)).
  \]
  Hence,
  $\kappa_\alpha(x) > \maxv{\kappa(\xbar + t[x - \xbar]),\, (1/\alpha) \|x - (\xbar + t[x - \xbar])\|} \ge
  \kappa_\alpha(x)$, which leads to a contradiction.  This proves (ii).

  Part~(iii). Since $\kappa$ is closed, the function
  $z\mapsto \max\set{\kappa(z),\,(1/\alpha)\|x - z\|}$ is closed and
  coercive for all $x$. Thus, $\pprox_{\alpha \kappa}(x)$ is nonempty
  for all $x$. Now, suppose that $\xbar$ and $\xhat$ belong to
  $\pprox_{\alpha\kappa}(x)$. Since the function
  $z\mapsto \max\set{\kappa(z),\,(1/\alpha)\|x - z\|}$ is convex, it
  follows that
  $\set{\lambda \xhat + (1-\lambda)\xbar\mid \lambda\in
    [0,1]}\subseteq \pprox_{\alpha \kappa}(x)$. Then from Part~(ii),
 \begin{equation}\label{eq3}
   \alpha\kappa_\alpha(x)
   = \|\xhat - x\|
   = \|\xbar - x\|
   = \|\lambda \xhat + (1-\lambda)\xbar-x\| \quad \forall \lambda\in (0,1).
 \end{equation}
 Thus,
 \[
   \alpha\kappa_\alpha(x)
   = \|\lambda \xhat + (1-\lambda)\xbar-x\|
   \le \lambda \|\xhat - x\| + (1-\lambda)\|\xbar - x\|
   = \alpha\kappa_\alpha(x).
 \]
 Hence, equality holds throughout. This implies that $\xbar - x$ and
 $\xhat - x$ differ by a nonnegative scaling, and thus we may assume
 without loss of generality that $\xbar - x = \tau(\xhat - x)$ for
 some $\tau \ge 0$.  Now, if $\kappa_\alpha(x) > 0$, we see
 from~\eqref{eq3} that $\tau = 1$, which implies $\xbar = \xhat$. On
 the other hand, if $\kappa_\alpha(x) = 0$, then~\eqref{eq3} gives
 $\xbar = x = \xhat$. Thus, $\pprox_{\alpha \kappa}(x)$ is a
 singleton.

 Let $\gamma > 0$. Then for any $x$,
 \begin{align*}
   \pprox_{\alpha\kappa}(\gamma x)
   &= \argmin_u \maxv{\kappa(u),\,(1/\alpha)\|u - \gamma x\|} \\
   &= \gamma \argmin_v \maxv{\kappa(\gamma v),\,(1/\alpha)\|\gamma v - \gamma x\|} \quad (u = \gamma v)\\
   &= \gamma \argmin_v \maxv{\kappa(v),\,(1/\alpha)\|v - x\|}= \gamma \pprox_{\alpha\kappa}(x),
 \end{align*}
 where the second last equality holds because $\kappa$ is positively
 homogeneous. Moreover, it is clear that
 $\pprox_{\alpha\kappa}(0)=0$. Thus, $\pprox_{\alpha\kappa}$ is
 positively homogeneous.

 We next prove continuity. Let $x_k\to x$ and write
 $\xbar_k = \pprox_{\alpha\kappa}(x_k)$ for notational
 simplicity. Then we have from the definition of $\xbar_k$ that for
 any $u$,
 \begin{equation}\label{optineq}
   \maxv{\kappa(\xbar_k),\,(1/\alpha)\|x_k - \xbar_k\|}
   \le \maxv{\kappa(u),\,(1/\alpha)\|x_k - u\|}.
 \end{equation}
 By setting $u=0$ in~\eqref{optineq}, we immediately conclude that
 $\{\xbar_k\}$ is bounded. Take any convergent subsequence
 $\{\xbar_{k_i}\}$ of $\{\xbar_k\}$ and let $\xhat$ denote its
 limit. Passing to the limit in~\eqref{optineq} along the subsequences
 $\{x_{k_i}\}$ and $\{\xbar_{k_i}\}$ and invoking the closedness of
 $\kappa$ gives
 \[
   \maxv{\kappa(\xhat),\,(1/\alpha)\|x - \xhat\|}
   \le \maxv{\kappa(u),\,(1/\alpha)\|x - u\|}
 \]
 for any $u$.  Thus, $\xhat = \pprox_{\alpha \kappa}(x)$. Since any
 convergent subsequence of the bounded sequence
 $\set{\pprox_{\alpha\kappa}(x_k)}$ converges to
 $\pprox_{\alpha \kappa}(x)$, it follows that
 $\pprox_{\alpha\kappa}(x_k)\to \pprox_{\alpha\kappa}(x)$. This proves
 the continuity of $\pprox_{\alpha\kappa}$.

 Part~(iv). Consider any $x$ satisfying $\kappa_\alpha(x) > 0$. Recall
 from Part~(ii) that $\|x - \xbar\| \ge \alpha\kappa(\xbar)$, where
 $\xbar = \pprox_{\alpha\kappa}(x)$ exists due to Part~(iii) and $\kappa$ being closed. In particular, this implies
  \[
  (1/\alpha)\|x - \xbar\| = \kappa_\alpha(x) > 0.
  \]
  Combining this with \cref{Prop_sub}, we see further that
  \begin{equation*}
  \begin{split}
    \partial \kappa_\alpha(x)
    &= \bigcup_{\mathclap{\mu\in [0,1]}}
    \:\left\{\partial(\lscale{\mu} \kappa)(\xbar) \medcap \frac{1-\mu}\alpha\left\{\frac{x-\xbar}{\|x - \xbar\|}\right\}
      \ \middle|\ \lscale{\mu}\kappa(\xbar) + \frac{1-\mu}{\alpha}\|x - \xbar\| = \kappa_\alpha(x)\right\}\\
    &\subseteq
    \left\{\lambda\cdot\frac{x-\xbar}{\|x - \xbar\|} \ \middle| \
      \lambda \in [0,1/\alpha]\right\}.
  \end{split}
  \end{equation*}
  Thus, for any elements $u$ and $v$ in $\partial \kappa_\alpha(x)$,
  there exist scalars $\lambda_1,\lambda_2\ge 0$ so that
  $u = \lambda_1(x-\xbar)$ and $v = \lambda_2(x-\xbar)$. In view of
  Part~(i) of this theorem,
  \[
  0<\kappa_\alpha(x) = \ip{x}{\lambda_1(x-\xbar)} = \ip{x}{\lambda_2(x-\xbar)},
  \]
  which establishes that $\ip{x}{x-\xbar} > 0$ and
  $\lambda_1=\lambda_2$. Hence, $\partial \kappa_\alpha(x)$ is a
  singleton whenever $\kappa_\alpha(x) >0$, which implies that
  $\kappa_\alpha$ is differentiable at those $x$.

  To obtain the formula for $\nabla \kappa_\alpha(x)$, note that
  $\nabla \kappa_\alpha(x) = \lambda(x-\xbar)$ for some
  $\lambda \ge 0$. In view of Part~(i), we must have
  $\ip{x}{\lambda(x-\xbar)} = \kappa_\alpha(x) > 0$, and thus
  \[
    \lambda = \frac{\kappa_\alpha(x)}{\ip{x}{x-\xbar}} = \frac{\|x - \xbar\|}{\alpha\ip{x}{x-\xbar}},
  \]
  where the second equality is due to Part~(ii).  This proves the formula
  of $\nabla \kappa_\alpha(x)$.
\end{proof}

\section{Computing polar envelopes: examples}\label{sec:computing_polar_env}

Let $\kappa$ be a closed gauge. We illustrate how to use
\cref{th:polar-envelope-subdiff} to compute
$\pprox_{\alpha \kappa}(x)$---and hence $\kappa_\alpha(x)$---at those
$x$ with $\kappa_\alpha(x) > 0$, which are points of
differentiability.

Recall from \cref{th:polar-envelope-subdiff}(iii) that the set
$\pprox_{\alpha\kappa}(x)$ is a singleton.  For any $x$ that satisfies
$\kappa_\alpha(x) > 0$, let $\ybar:=\pprox_{\alpha\kappa}(x)$ and
write $\rbar := \kappa_\alpha(x)> 0$. Then
$\kappa(\ybar) \le (1/\alpha)\|\ybar - x\| = \rbar$ according to
\cref{th:polar-envelope-subdiff}(ii). In particular, we have
$\ybar \neq x$.

We consider two cases. First, suppose that
$\kappa(\ybar) < (1/\alpha)\|\ybar - x\|$. It follows from
\citet[Corollary~2.8.15]{Za02} that the optimality conditions
for~\eqref{eq:10} are given by
\[
  0\in \frac1\alpha\cdot\frac{\ybar - x}{\|\ybar - x\|}
  + \normal(\ybar\mid \dom\kappa).
\]
This implies $\ybar = \proj_{\dom\kappa}(x)$, and in particular, that
the projection exists in this case.

On the other hand, suppose that
\begin{equation}\label{plug_in0}
  \kappa(\ybar) = (1/\alpha)\|\ybar - x\| = \rbar.
\end{equation}
Again applying \citet[Corollary~2.8.15]{Za02} to obtain the optimality
condition for~\eqref{eq:10}, we deduce that there exists
$\lambda \in [0,1]$ with the property that
\begin{equation}\label{eq:hatimes6}
  0\in \frac{1-\lambda}\alpha\cdot\frac{\ybar - x}{\|\ybar - x\|} + \partial(\lscale{\lambda}\kappa)(\ybar).
\end{equation}
We claim that $\lambda \neq 1$. Suppose to the contrary that
$\lambda = 1$. Then~\eqref{eq:hatimes6} becomes
$0 \in \partial\kappa(\ybar)$. This means that $\ybar$ minimizes the
gauge function $\kappa$, giving $\kappa(\ybar) = 0$. This
contradicts~\eqref{plug_in0} because $\bar r > 0$.

  Thus, $\lambda \neq 1$ and we obtain from~\eqref{eq:hatimes6} that
  \[
  0 \in \ybar - x + \frac{\alpha}{1-\lambda}\|\ybar - x\| \partial(\lscale{\lambda}\kappa)(\ybar) \subset \ybar - x + \normal(\ybar \mid [\kappa\le \rbar]),
  \]
  where the second inclusion follows from
  \citet[Corollary~2.9.5]{Za02}, \eqref{plug_in0}, and the fact that
  $\rbar > 0$. This implies $\ybar = \proj_{[\kappa\le
    \rbar]}(x)$. Substituting this relation into~\eqref{plug_in0}
  shows that $\rbar$ satisfies the equation
  \begin{equation}\label{hahahahahaha}
    \alpha^2 \rbar^2
    = \| x - \proj_{[\kappa\le\rbar]}(x)\|^2.
  \end{equation}
  Next, note that the function $r\mapsto \|x - \proj_{[\kappa\le r]}(x)\|^2$ is nonincreasing on $(0,\infty)$: indeed, for $s\ge r > 0$, we have
  \[
  \begin{aligned}
  & \|x - \proj_{[\kappa\le r]}(x)\|^2  = \|x - \proj_{[\kappa\le s]}(x) + \proj_{[\kappa\le s]}(x) -\proj_{[\kappa\le r]}(x)\|^2\\
  & = \|x - \proj_{[\kappa\le s]}(x)\|^2 + 2 \ip{x - \proj_{[\kappa\le s]}(x)}{\,\proj_{[\kappa\le s]}(x) -\proj_{[\kappa\le r]}(x)}\\
  &\quad\ + \|\proj_{[\kappa\le s]}(x) -\proj_{[\kappa\le r]}(x)\|^2\\
  & \overset{\rm(a)}\ge \|x - \proj_{[\kappa\le s]}(x)\|^2 + \|\proj_{[\kappa\le s]}(x) -\proj_{[\kappa\le r]}(x)\|^2\ge \|x - \proj_{[\kappa\le s]}(x)\|^2 ,
  \end{aligned}
  \]
  where the inequality (a) follows from the property of projections
  and the fact that $\proj_{[\kappa\le r]}(x)\in [\kappa\le s]$.
  Consequently, the function
\[
r\mapsto \alpha^2r^2 - \|x - \proj_{[\kappa\le r]}(x)\|^2
\]
is \emph{strictly} increasing on $(0,\infty)$. Thus, $\bar r$ is the \emph{unique positive} root satisfying~\eqref{hahahahahaha}.

In summary,
\begin{align*}
   \pprox_{\alpha\kappa}(x) &=
  \begin{cases}
    \proj_{\dom\kappa}(x) & \mbox{if $\ybar:=\proj_{\dom\kappa}(x)$ exists and $\kappa(\ybar) < (1/\alpha) \|x - \ybar\|$,}
    \\\proj_{[\kappa\le \rbar]}(x) & \mbox{otherwise,}
  \end{cases}
  \\[4pt]
  \kappa_\alpha(x) &=
  \begin{cases}
    (1/\alpha)\|x-\ybar\| & \mbox{if $\ybar:=\proj_{\dom\kappa}(x)$ exists and $\kappa(\ybar) < (1/\alpha) \|x - \ybar\|$,}
    \\\rbar & \mbox{otherwise,}
  \end{cases}
\end{align*}
where $\rbar$ is the unique positive root of equation~\eqref{hahahahahaha}.

We now give examples that show how these formulas specialize to
common cases.

\begin{example}[Linear function over a cone]\label{ex_lin_ind}
  Let $\kappa(z) = \ip{c}{z} + \delta\Ks(z)$ for some closed convex
  cone $\Kscr$ and some vector $c$ in the dual cone $\Kscr^*$. Then
  $\kappa$ is a closed gauge, $\dom\kappa = \Kscr$, and $\proj\Ks(x)$
  exists for all $x$ because $\Kscr$ is closed. For any $x$ satisfying
  $\kappa_\alpha(x) > 0$, the polar proximal map is given by
  \begin{align*}
    \pprox_{\alpha\kappa}(x) &=
    \begin{cases}
      \proj\Ks(x) & \mbox{if $\ip{c}{\proj\Ks(x)} < (1/\alpha) \|\proj\Ksp(x)\|$,}
    \\\proj_{\scriptscriptstyle\widebar\Kscr(\rbar)}(x) & \mbox{otherwise,}
  \end{cases}
    \\
    \kappa_\alpha(x) &=
    \begin{cases}
      (1/\alpha)\|\proj\Ksp(x)\| & \mbox{if $\ip{c}{\proj\Ks(x)} < (1/\alpha) \|\proj\Ksp(x)\|$,}
    \\\rbar & \mbox{otherwise,}
  \end{cases}
  \end{align*}
  where $\widebar\Kscr(\rbar):=\Kscr\cap\set{u|\ip{c}{u} \le \rbar}$
  and $\rbar$ is the unique positive root of the equation
\begin{equation*}
    \alpha^2 \rbar^2
    = \| x - \proj_{\scriptscriptstyle\widebar\Kscr(\rbar)}(x)\|^2.
  \end{equation*}
  Here we used the Moreau identity to determine that
  $x=\proj\Ks(x)+\proj\Ksp(x)$. \QEDA
\end{example}

\begin{example}[Continuous gauge]\label{example_cont_gauge}
  Let $\kappa$ be a continuous gauge. Then we have from
  \cref{th:polar-envelope-subdiff}(ii) that
  $\kappa(\pprox_{\alpha\kappa}(x)) =
  (1/\alpha)\|\pprox_{\alpha\kappa}(x) - x\|$. Then, for any $x$
  satisfying $\kappa_\alpha(x) > 0$, it holds that
  \[
    \kappa_\alpha(x) = \rbar
    \quad\text{and}\quad
    \pprox_{\alpha\kappa}(x) = \proj_{[\kappa\le\rbar]}(x),
  \]
  where $\rbar$ is the unique positive root of the equation
\begin{equation}\label{example_eq_1}
  \alpha^2\rbar^2
  = \|x - \proj_{[\kappa\le\rbar]}(x)\|^2.
\end{equation}
\QEDA
\end{example}

\begin{example}[Infinity norm]
  As a concrete example, consider $\kappa = \|\cdot\|_\infty$. Note
  that $\kappa_\alpha(x) = 0$ if and only if $x = 0$. For $x\neq 0$,
  equation~\eqref{example_eq_1} becomes
  \[
  \alpha^2 \rbar^2 = \sum_{i=1}^n (|x_i| - \rbar)_+^2.
  \]
  Letting $\bar \gamma = \rbar^{-1}$, the above equation is equivalent
  to the equation
  \[
  \alpha^2 = \sum_{i=1}^n (\bar \gamma|x_i| - 1)_+^2.
  \]
  This is a piecewise linear quadratic equation with exactly one positive root
  because the function on the right-hand side is zero for $\bar\gamma \in (0,1/\|x\|_\infty]$ and is strictly increasing on $(1/\|x\|_\infty,\infty)$, mapping this interval to $(0,\infty)$. The equation
  can be solved in $O(n\log n)$ time. Once $\rbar$ is found, we can
  obtain $\pprox_{\alpha\kappa}(x)=\proj_{[\|\cdot\|_\infty\le \rbar]}(x)$.
  \QEDA
\end{example}

\section{Constructing smooth gauge dual problems}\label{sec:smooth_dual}

The polar envelope can be naturally incorporated in gauge optimization
problems to yield smooth gauge dual problems. We show how a primal
solution can be recovered after solving the smooth dual problem.
First we collect in the following proposition some essential
properties of the polar proximal map of a continuous gauge.

\begin{proposition}[Polar proximal map of a continuous
  gauge]\label{property_2}
  Let $\kappa$ be a continuous gauge and $\alpha > 0$. Then
  \begin{equation}\label{formula}
  \pprox_{\alpha\kappa}(x) = \proj_{[\kappa\le\kappa_\alpha\!(x)]}(x).
  \end{equation}
  Moreover, the following hold.
  \begin{enumerate}[{\rm (i)}]
    \item $\|\pprox_{\alpha\kappa}(x)\|\le \|x\|$ for all $x$.
    \item For any $\beta> 0$ and $M > 0$, the function
      $x\mapsto \pprox_{\alpha\kappa}(x)$ is globally Lipschitz
      continuous in the set
      $\Xi_{M,\beta}:=\set{x | \|x\|\le M, \kappa_\alpha(x)\ge
        \beta}$. Specifically, it holds that
        \[
        \|\pprox_{\alpha \kappa}(x) - \pprox_{\alpha \kappa}(y)\|\le \frac{3M}{\alpha\beta}\|x-y\|
        \]
        for any $x$ and $y$ in $\Xi_{M,\beta}$.
    \item At any $x$ satisfying $\kappa_\alpha(x) > 0$, the function $\nabla \kappa_\alpha$ is locally Lipschitz.
  \end{enumerate}
\end{proposition}
\begin{proof}
  We first prove~\eqref{formula}. In view of
  \cref{example_cont_gauge}, it suffices to show that~\eqref{formula}
  holds also in the case when $\kappa_\alpha(x) = 0$. Fix any such
  $x$. Write $\ybar = \pprox_{\alpha\kappa}(x)$. Then we have
  $\kappa(\ybar) = 0$ and $\|\ybar - x\| = 0$. Consequently, we have
  $\ybar = x$ and that $\kappa(x) = 0$. In particular, this indicates
  that we can write
  $\ybar = x = \proj_{[\kappa\le 0]}(x) = \proj_{[\kappa\le
    \kappa_\alpha(x)]}(x)$. This proves~\eqref{formula}.

  We now prove Part~(i). In view of~\eqref{formula} and the fact that
  $\kappa(0)\le\kappa_\alpha(x)$ for any $x$, we have from the
  definition of projection that
  \[
    \ip{x-\pprox_{\alpha\kappa}(x)}{\,0-\pprox_{\alpha\kappa}(x)}\le 0,
  \]
  which implies that
  $\|\pprox_{\alpha\kappa}(x)\|^2 \le
  \ip{x}{\pprox_{\alpha\kappa}(x)}$.  The conclusion of Part~(i) now
  follows from this and the Cauchy-Schwartz inequality.

  Next, in view of~\eqref{formula} and the nonexpansiveness of
  projections onto closed convex sets, it follows that
  \begin{equation}\label{nonexpansive_1}
  \|\pprox_{\alpha\kappa}(x) - \pprox_{\alpha\kappa}(y)\|\le \|x - y\|
  \end{equation}
  whenever $\kappa_\alpha(x) = \kappa_\alpha(y)$. Now, consider any $x$, $y\in \Xi_{M,\beta}$. Thus,
  \[
  \begin{aligned}
     &\|\pprox_{\alpha\kappa}(x) - \pprox_{\alpha\kappa}(y)\|\\
     &\le \|\pprox_{\alpha\kappa}(x) -\pprox_{\alpha\kappa}(\kappa_\alpha(x)y/\kappa_\alpha(y))\| + \|\pprox_{\alpha\kappa}(\kappa_\alpha(x)y/\kappa_\alpha(y)) - \pprox_{\alpha\kappa}(y)\|\\
     &= \|\pprox_{\alpha\kappa}(x) -\pprox_{\alpha\kappa}(\kappa_\alpha(x)y/\kappa_\alpha(y))\| + \|\kappa_\alpha(x)\pprox_{\alpha\kappa}(y/\kappa_\alpha(y)) - \pprox_{\alpha\kappa}(y)\|\\
     &\le \|x - \kappa_\alpha(x)y/\kappa_\alpha(y)\| + \|\kappa_\alpha(x)\pprox_{\alpha\kappa}(y/\kappa_\alpha(y)) - \pprox_{\alpha\kappa}(y)\|\\
     & = \kappa_\alpha(y)^{-1}\big[\,\|\kappa_\alpha(y)x - \kappa_\alpha(x)y\| + \|\pprox_{\alpha\kappa}(y)\|\cdot|\kappa_\alpha(y) - \kappa_\alpha(x)|\,\big],
  \end{aligned}
  \]
  where the equalities follow from the positive homogeneity of
  $\pprox_{\alpha\kappa}$
  (cf.~\cref{th:polar-envelope-subdiff}[iii]), and the second
  inequality follows from~\eqref{nonexpansive_1}. Next, using Part~(i),
  \[
  \begin{aligned}
    \|\pprox_{\alpha\kappa}(x) - \pprox_{\alpha\kappa}(y)\|
     & \le \kappa_\alpha(y)^{-1}\big[\,\|\kappa_\alpha(y)x - \kappa_\alpha(x)y\| + \|y\|\cdot|\kappa_\alpha(y) - \kappa_\alpha(x)|\,\big]\\
     & \overset{\rm(a)}{\le} \kappa_\alpha(y)^{-1}\big[\,\|\kappa_\alpha(y)(x-y)\| + 2\|y\|\cdot|\kappa_\alpha(x) - \kappa_\alpha(y)|\,\big]\\
     & \overset{\rm(b)}{\le} \beta^{-1}\big[\,\|\kappa_\alpha(y)(x-y)\| + 2M|\kappa_\alpha(x) - \kappa_\alpha(y)|\,\big]\\
     & \overset{\rm(c)}{\le} \frac{3M}{\alpha\beta}\|x-y\|.
  \end{aligned}
  \]
  where (a) follows from the triangle inequality, (b) follows from the
  definition of $\Xi_{M,\beta}$, and (c) follows from the Lipschitz
  continuity of $\kappa_\alpha$ (cf.~\cref{known}[i]). This
  proves global Lipschitz continuity on $\Xi_{M,\beta}$. Finally, the
  conclusion in Part~(iii) follows immediately from this and the
  formula of $\nabla \kappa_\alpha$ given in
  \cref{th:polar-envelope-subdiff}(iv).
\end{proof}

\subsection{Sublinear regularization}

\cref{property_2} suggests a natural smoothing strategy for the following gauge optimization problem
\begin{subequations} \label{eq:36}
\begin{equation} \label{eq:bp0}
  \minimize_{x\in\Xscr}\quad \kappa(x) \quad\st\quad  \rho(b - Ax)\le \sigma,
\end{equation}
where $\kappa$ and $\rho$ are both closed gauges and
$A:\Xscr \to \Yscr$ is a linear map. We assume that
$\sigma \in [0,\rho(b))$, that $\kappa^{-1}(0)=\set{0}$ and $\rho^{-1}(0) = \set{0}$, and that the data satisfy the constraint qualification
\begin{equation}\label{CQ}
  {\rm ri}\,\dom \kappa\cap A^{-1}{\rm ri}\,{\cal C} \neq \emptyset.
\end{equation}
\end{subequations}
Here, we define ${\cal C} := \set{u\mid \rho(b - u)\le \sigma}$.

Consider the following regularization of~\eqref{eq:bp0}
\begin{equation} \label{eq:bp}
  \minimize_{x\in\Xscr}\quad \kappa(x)+\alpha\|x\| \quad\st\quad  \rho(b - Ax)\le \sigma,
\end{equation}
where $\alpha$ is a positive regularization parameter. We expect that, for small values of $\alpha$, solutions of the
perturbed problem are good approximations to solutions of the original
problem~\eqref{eq:36}. In particular, because the objective
$\kappa + \alpha\|\cdot\|$ epi-converges to $\kappa$ \cite[Proposition
7.4(c)]{rtrw:1998}, it follows that cluster points (if they exist) of
solutions of~\eqref{eq:bp} are also minimizers of~\eqref{eq:36}; see
\cite[Theorem 7.31(b)]{rtrw:1998}.

The objective of the perturbed problem~\eqref{eq:bp} is in general a
nonsmooth gauge. As we demonstrate below, however, the gauge dual has
a smooth objective.

\subsection{Primal and dual pairs under convolution}

The corresponding gauge dual is given by
\begin{equation} \label{eq:gdbp}
  \minimize_{y\in\Yscr}\quad (\kappa^\circ \pconv (1/\alpha)\|\cdot\|)(A^*y) \quad\st\quad  \ip{b}{y} - \sigma\rho^\circ(y)\ge 1,
\end{equation}
where the objective in~\eqref{eq:gdbp} follows from
\citet[Section~1.1]{FriedlanderMacedoPong:2014} and an application of
\cref{lem1} with $\kappa_1 = \kappa^\circ$ and
$\kappa_2 = (1/\alpha)\|\cdot\|$. The next result concerning the gauge
dual pair~\eqref{eq:bp} and~\eqref{eq:gdbp} establishes that the dual
problem is in some sense smooth, and gives a formula for the
relationship between the primal and dual solutions.

\begin{theorem}[Primal and dual problems under polar convolution]\label{thm1.5}
  Consider problems~\eqref{eq:bp} and~\eqref{eq:gdbp}, where $\kappa$, $\rho$, $A$, $b$ and $\sigma$ are as in~\eqref{eq:36}. Then the following conclusions hold.
  \begin{enumerate}[{\rm (i)}]
    \item The optimal values of~\eqref{eq:bp} and~\eqref{eq:gdbp} are finite, positive, and are attained. Moreover, they are reciprocal of each other.
    \item The objective of~\eqref{eq:gdbp} is smooth with a locally Lipschitz gradient on the feasible set of~\eqref{eq:gdbp}.
    \item Let $\ybar$ be an optimal solution of~\eqref{eq:gdbp} and $\rbar$ be its optimal value. Then
    \begin{equation}\label{hatx}
    \xhat = \frac{1/\rbar}{\kappa(\prox_{\rbar\kappa}(A^*\ybar))+\alpha\|\prox_{\rbar\kappa}(A^*\ybar)\|}\prox_{\rbar\kappa}(A^*\ybar)
    \end{equation}
    is an optimal solution of~\eqref{eq:bp}.
  \end{enumerate}
\end{theorem}
\begin{proof}
  Note from~\eqref{CQ} that we have
  \[
  {\rm ri}\,\dom \left(\kappa +\alpha\|\cdot\|\right)\cap A^{-1}{\rm ri}\,{\cal C} = {\rm ri}\,\dom \kappa\cap A^{-1}{\rm ri}\,{\cal C}\neq \emptyset.
  \]
  Also, from \cref{known}(i), we trivially have
  ${\rm ri}\,\dom (\kappa^\circ\pconv (1/\alpha)\|\cdot\|)\cap A^*{\rm
    ri}\,{\cal C}' = A^*{\rm ri}\,{\cal C}'\neq \emptyset$, where
  ${\cal C}'$ is the antipolar set of ${\cal C}$, which is a nonempty
  set because $0\notin \cal C$, thanks to the assumption that
  $\rho(b)> \sigma$. In view of these and
  \citet[Corollary~5.6]{FriedlanderMacedoPong:2014}, the optimal value
  of~\eqref{eq:gdbp} is the reciprocal of the optimal value
  of~\eqref{eq:bp}, and both optimal values are finite, positive, and
  attained.

  Next, since $\kappa^{-1}(0)=\set{0}$, it follows that $\kappa^\circ$ is
  continuous. Moreover, because the optimal value of~\eqref{eq:gdbp}
  is attained and is positive, it follows that
  $(\kappa^\circ\pconv (1/\alpha)\|\cdot\|)(A^*y) > 0$ for any
  $y$ feasible for~\eqref{eq:gdbp}. The local Lipschitz
  continuity of the gradient of the objective of~\eqref{eq:gdbp} now
  follows from this and \cref{property_2}(iii).

  We now prove Part~(iii). Let $\ybar$ be an optimal solution
  of~\eqref{eq:gdbp} and let $\bar u$ satisfy
\begin{equation}\label{eq:haha}
  \kappa^\circ(\bar u) = (1/\alpha)\|A^*\ybar - \bar u\| = (\kappa^\circ \pconv (1/\alpha)\|\cdot\|)(A^*\ybar) =: \rbar > 0,
\end{equation}
which exists according to \cref{th:polar-envelope-subdiff}(ii)
and (iii), and the continuity of $\kappa^\circ$. Let $\xhat$ be a
solution of~\eqref{eq:bp} and let $\widehat v$ be such that
$b = A\xhat + \widehat v$ and $\rho(\widehat v) \le \sigma$. Then we have from
strong duality of the gauge dual pairs~\eqref{eq:bp} and
\eqref{eq:gdbp} that
\begin{equation}\label{eq:3}
\begin{aligned}
  1 & = (\kappa(\xhat) + \alpha\|\xhat\|)\cdot (\kappa^\circ \pconv (1/\alpha)\|\cdot\|)(A^*\ybar)\\
  & \overset{\rm(a)}= \kappa(\xhat)\cdot \kappa^\circ(\bar u) + \|\xhat\|\cdot \|A^*\ybar - \bar u\|\\
  & \overset{\rm(b)}\ge \ip{\xhat}{\bar u} + \ip{\xhat}{A^*\ybar - \bar u} = \ip{\xhat}{A^*\ybar} \\
  & \overset{\rm(c)}= \ip{b - \widehat v}{\ybar}
    \overset{\rm(d)}\ge \ip{b}{\ybar}-\sigma\rho^\circ(\ybar)\ge 1,
  \end{aligned}
  \end{equation}
where (a) follows from~\eqref{eq:haha}, (b) follows from~\eqref{eq:9},
(c) follows from the fact that $A\xhat + \widehat v = b$, and (d) follows
from \citet[Proposition~2.1(iii)]{FriedlanderMacedoPong:2014} (for
$\sigma > 0$) and the continuity of $\rho^\circ$ (for $\sigma = 0$),
thanks to the assumption that $\rho^{-1}(0)=\set{0}$. Thus, equality holds throughout
the above relation, and we have in particular that
\[
  \|\xhat\|\cdot\|A^*\ybar- \bar u\|
  = \ip{\xhat}{A^*\ybar - \bar u}.
\]
Since $\xhat \neq 0$, this implies that
$\xhat = \gamma (A^*\ybar - \bar u)$ for some $\gamma > 0$. Next,
combine~\eqref{eq:haha} with the first equation in~\eqref{eq:3} to
obtain $\kappa(\xhat) + \alpha\|\xhat\| = \rbar^{-1}$. Together with
the expression that we just derived for $\xhat$, we deduce that
\[
  \gamma = \frac{1/\rbar}{\kappa(A^*\ybar - \bar u) + \alpha\|A^*\ybar - \bar u\|}.
\]
Recall from~\eqref{formula} that
$\bar u = \proj_{\kappa^\circ(\cdot)\le \rbar}(A^*\ybar)$ because
$\kappa^\circ$ is continuous. Thus, by the Moreau identity,
$A^*\ybar - \bar u = \prox_{\rbar\kappa}(A^*\ybar)$. Then we can
compute $\xhat$ as~\eqref{hatx}.
\end{proof}

As an immediate application of \cref{thm1.5}, we consider the
basis pursuit problem \citep{chendonosaun:2001}, which takes the form
\begin{equation} \label{eq:bp00}
  \minimize_{x\in\R^n}\quad \|x\|_1 \quad\st\quad  Ax = b,
\end{equation}
where $A\in \R^{m\times n}$, $A^{-1}\set{b}\neq \emptyset$, and
$b\neq 0$. This is just~\eqref{eq:bp0} with $\kappa = \|\cdot\|_1$,
$\rho = \delta_{\{0\}}$, and $\sigma = 0$, and it is routine to check
that the assumptions on $\kappa$, $\rho$, $A$, $b$ and $\sigma$ for
\eqref{eq:36} are satisfied. For each $\alpha>0$, one can consider the
regularization of~\eqref{eq:bp00}
\begin{equation} \label{eq:bp1}
  \minimize_{x\in\R^n}\quad \|x\|_1+\alpha\|x\| \quad\st\quad  Ax = b,
\end{equation}
and its gauge dual problem
\begin{equation} \label{eq:gdbp1} \minimize_{y\in\R^m}\quad
  (\|\cdot\|_\infty \pconv (1/\alpha)\|\cdot\|)(A^*y) \quad\st\quad
  \ip{b}{\lambda} \ge 1.
\end{equation}
We thus have the following immediate corollary of \cref{thm1.5}.

\begin{corollary}[Polar smoothing of gauge dual]\label{cor3}
  Consider the gauge dual pair~\eqref{eq:bp1} and~\eqref{eq:gdbp1}
  with $A^{-1}\set{b}\neq \emptyset$ and $b\neq 0$. Then the following
  conclusions hold.
  \begin{enumerate}[{\rm (i)}]
    \item The optimal values of~\eqref{eq:bp1} and~\eqref{eq:gdbp1} are finite, positive and are attained. Moreover, they are reciprocal of each other.
    \item The objective of~\eqref{eq:gdbp1} is smooth with a locally Lipschitz gradient on the feasible set of~\eqref{eq:gdbp1}.
    \item Let $\bar \lambda$ be an optimal solution of~\eqref{eq:gdbp1} and $\rbar$ be its optimal value. Then
    \begin{equation}\label{hatx1}
      \xhat = \frac{1/\rbar}{\|\prox_{\rbar\|\cdot\|_1}(A^*\ybar )\|_1+\alpha\|\prox_{\rbar\|\cdot\|_1}(A^*\ybar)\|}\prox_{\rbar\|\cdot\|_1}(A^*\ybar)
    \end{equation}
    is an optimal solution of~\eqref{eq:bp1}.
  \end{enumerate}
\end{corollary}

Corollary~\ref{cor3}(ii) suggests that, in order to solve~\eqref{eq:bp1}, one can apply gradient descent algorithm with line
search to solve~\eqref{eq:gdbp1}, which is appropriate because the gradient is locally Lipschitz. A solution of~\eqref{eq:bp1} is then
recovered via~\eqref{hatx1}.

As we discussed in connection with the duality correspondences shown
in Figs.~\ref{fig:fenchel} and~\ref{fig:gauge}, we contrast the above
approach with the usual way of constructing smooth {\em Lagrange} dual
problems by adding a {\em strongly convex} term to the primal
objective. In this latter approach, for the functions and data
in~\eqref{eq:36} and $\alpha > 0$, one considers the following
approximation of~\eqref{eq:bp0}:
\begin{equation} \label{eq:bp_L}
  \minimize_{x\in\Xscr}\quad \kappa(x)+(\alpha/2)\|x\|^2 \quad\st\quad  \rho(b - Ax)\le \sigma.
\end{equation}
The Lagrange dual problem of~\eqref{eq:bp_L} takes the form
\begin{equation} \label{eq:Ldbp}
  \maximize_{y\in\Yscr}\quad \ip{b}{y} - (1/2\alpha)\dist^2_{[\kappa^\circ \le 1]}(A^*y) - \sigma\rho^\circ(y).
\end{equation}
This dual problem is the sum of the nonsmooth function
$y\mapsto - \sigma\rho^\circ(y)$ and a smooth function with Lipschitz
gradient, and thus can be suitably solved by first-order methods such
as the proximal-gradient algorithm. The solution $\bar x$
of~\eqref{eq:bp_L} is then recovered from a solution $\ybar$
of~\eqref{eq:Ldbp} via
\[
\bar x = \prox_{(1/\alpha)\kappa}(\alpha^{-1}A^*\ybar).
\]
In contrast to the gauge dualization approach based on the pair~\eqref{eq:bp} and~\eqref{eq:gdbp}, the Lagrange dual pair~\eqref{eq:bp_L} and~\eqref{eq:Ldbp} are not gauge optimization
problems, even though the original problem~\eqref{eq:bp0} that we are
approximating is a gauge optimization problem. Moreover, while the
objective of~\eqref{eq:bp_L} is strongly convex so that the objective
of its Lagrange dual~\eqref{eq:Ldbp} has a smooth component, it is
interesting to note that the objective of~\eqref{eq:bp} is {\em not
  strongly convex} but its gauge dual~\eqref{eq:gdbp} still has an
objective that is smooth in the feasible region.

\section{Proximal-point-like algorithms based on polar envelopes}
\label{sec:polar-prox-algo}

In this section, we discuss proximal-point-like algorithms for solving
the following general convex optimization problem
\begin{subequations} \label{eq:37}
\begin{equation} \label{eq:1}
  \minimize_{x\in\Xscr}\enspace f(x),
\end{equation}
where $f:\Xscr \to \Rpi$ is a proper closed nonnegative convex function with
\begin{equation} \label{eq:1b}
  \inf f > 0\ \ \ \mbox{and}\ \ \ \argmin f\neq \emptyset.
\end{equation}
\end{subequations}

Our proximal-point-like algorithms are based on a specific polar
envelope and the corresponding polar proximal map. To proceed, we
first rewrite~\eqref{eq:1} as the equivalent gauge
optimization problem
\begin{equation*}
  \minimize_{x\in\Xscr,\,\lambda \ge 1}\enspace
  f^\pi(x,\lambda),
\end{equation*}
where the perspective function $f^\pi$ is defined in~\eqref{eq:12}.

\subsection{The projected polar proximal map}
We introduce two important ingredients for the development of our proximal-point-like algorithms, namely the {\em projected} polar envelope and the {\em projected} polar proximal map of $f^\pi(x,\lambda)$: for any $\alpha > 0$,
\[
\begin{aligned}
  \penv(x)
  &:= f^\pi_\alpha(x,1),
    \\
  \rpprox(x)
  &:= \pprox_{\alpha f^\pi}(x,1),
\end{aligned}
\]
where $f^\pi_\alpha$ is shorthand for the polar envelope of the
perspective transform, i.e., $f^\pi_\alpha\equiv(f^\pi)_\alpha$.  Note
that $\rpprox(x)$ is a singleton for any $x$ because $\alpha f^\pi$ is
closed and by virtue of
\cref{th:polar-envelope-subdiff}(iii). We show that the set of
minimizers of the projected polar envelope and the set of ``fixed
points'' of the polar proximal map are closely related to the set of
minimizers of the original problem~\eqref{eq:1}.

We need the following auxiliary lemma concerning $\dom f^\pi$. For
this section only, for any convex set $\Cscr$ we define its negative
polar cone by $\Cscr^-=(\cone\Cscr)^\circ$.
\begin{lemma}
  Consider $f$ as in~\eqref{eq:37}. Then it holds that
  \begin{equation}\label{domf}
  \cl \dom f^\pi = \cl \left(\bigcup_{\lambda > 0}\lambda(\dom f\times \{1\})\right).
\end{equation}
Moreover, for any $(\bar x,\bar\lambda)\in \dom f^\pi$, it holds that
  \begin{equation}\label{normalcone}
    \normal((\xbar,\bar \lambda)\mid\dom f^\pi)
    = \set{w \in [\dom f\times \{1\}]^-\mid
      \ip{w}{(\xbar,\bar\lambda)} = 0}.
  \end{equation}
\end{lemma}
\begin{proof}
  First of all, we note from the definition of the recession cone that
  for any closed convex set $\Cscr$ and linear map $A$ so that
  $A(\Cscr)$ is well defined, one has
\[
A(\Cscr_\infty) \subseteq [\cl A(\Cscr)]_\infty.
\]
Applying this with $\Cscr = \epi f$ and $A = \proj_\Xscr$ and invoking $\epi \rec{f} = (\epi f)_\infty$ gives
\begin{equation}\label{dominclusion}
\dom \rec{f} = \proj_\Xscr(\epi \rec{f}) \subseteq [\cl \proj_\Xscr(\epi f)]_\infty = [\cl \dom f]_\infty.
\end{equation}
Next, it is not hard to see directly from the definition of $f^\pi$ that
\begin{equation}\label{domfpre}
\begin{aligned}
  \dom f^\pi & = (\dom \rec{f} \times \{0\})\cup \bigcup_{\lambda > 0}\lambda(\dom f\times \{1\}).
\end{aligned}
\end{equation}
Then we deduce further that
\begin{equation*}
  \begin{aligned}
    \cl\dom f^\pi&\supseteq\cl \left(\bigcup_{\lambda > 0}\lambda(\dom f\times \{1\})\right)
    \overset{\rm(a)}=\cl \left(\bigcup_{\lambda > 0}\lambda(\cl\dom f\times \{1\})\right)\\
     & \overset{\rm(b)}= ([\cl\dom f]_\infty \times \{0\})\cup \bigcup_{\lambda > 0}\lambda(\cl\dom f\times \{1\})
      \overset{\rm(c)}\supseteq \dom f^\pi,
  \end{aligned}
\end{equation*}
where (a) follows directly from the definition of closure, (b) follows
from \citet[Lemma~2.1.1]{AuT03}, and (c) follows
from~\eqref{dominclusion} and~\eqref{domfpre}. Taking the closure on
both sides of the above inclusion establishes~\eqref{domf}.

Next, using~\eqref{domf} and the definition of negative polar, we have
\begin{equation*}
    \begin{aligned}
      [\dom f^\pi]^- = \left[\cl \left(\bigcup_{\lambda > 0}\lambda(\dom f\times \{1\})\right)\right]^- = [\dom f\times \{1\}]^-.
    \end{aligned}
  \end{equation*}
  The relation~\eqref{normalcone} now follows from this, the
  definition of normal cones, and the fact that $\dom f^\pi$ is a
  cone.
\end{proof}

In our next theorem, we study the relationship between the minimizers
of the projected polar envelope $\penv$ and those of the original
problem~\eqref{eq:1}. This result depends on the subdifferential of
the perspective function, characterized below. We defer to
\citet[Proposition~2.3(v)]{Com18} and \citet[Lemma~3.8]{ABDFM18}
for proof.

\begin{proposition}[Subdifferential of perspective]\label{subdiff_pers}
  Suppose that $f:\Xscr\to \Rpi$ is a proper closed nonnegative convex
  function. Then for fixed $(x,\lambda)\in\dom f^\pi$,
  \[
  \partial f^\pi(x,\lambda) = \begin{cases}
    \set{(u,-f^*(u))\mid u \in \partial f(\lambda^{-1}x)} & \mbox{if $\lambda > 0$},\\
    \set{(u,\gamma)\mid (u,-\gamma)\in \epi f^*, u \in \partial \rec{f}(x)} & \mbox{if $\lambda = 0$}.
  \end{cases}
  \]
\end{proposition}

\begin{theorem}[Minimizers of the projected polar envelope]\label{Thm1}
  Consider $f$ as in~\eqref{eq:37} and let $\alpha > 0$.
  \begin{enumerate}[{\rm (i)}]
  \item Suppose $x\in \argmin \penv$ and let
    $(\xbar,\bar \lambda) = \rpprox(x)$. Then $\xbar = x$,
    $\bar \lambda > 0$, $\bar \lambda^{-1}x\in \argmin f$.
  \item If $x\in \argmin f$, then $\lambda x\in \argmin \penv$, where
    $\lambda:=[1+\alpha f(x)]^{-1}$.
  \end{enumerate}
\end{theorem}

\begin{proof}
  We first derive a formula for the subdifferential of $\penv$. To
  this end, define
\begin{equation}\label{f1f2}
  f_1 =(1/\alpha)\|\cdot\|\ \quad\mbox{and}\quad \ f_2 = f^\pi.
\end{equation}
Then $f_1\pconv f_2$ is continuous thanks to \cref{cor:bounded}, and $\penv(x) = (f_1\pconv f_2)(x,1)$. We thus obtain the following formula for the subdifferential:
  \begin{equation}\label{sub_formula2}
  \partial \penv(x)= \set{u\mid \exists \beta\ \ {\rm s.t.}\ \ (u,\beta)\in \partial(f_1\pconv f_2)(x,1)};
  \end{equation}
here, $\partial(f_1\pconv f_2)(x,1)$ is given by \cref{Prop_sub} as
  \begin{equation}\label{sub_formula}
  \begin{aligned}
  \partial (f_1\pconv f_2)(x,1) = \bigcup_{\mu\in \Gamma} \Set{\partial(\mu f_1)(x - \xbar,1 - \bar \lambda)\cap \partial(\lscale{[1-\mu]}f_2)(\xbar,\bar \lambda)}.
  \end{aligned}
  \end{equation}
  where $\Gamma := \set{\mu\in [0,1]\mid \mu f_1(x - \xbar,1 - \bar \lambda) + \lscale{(1-\mu)} f_2(\xbar,\bar \lambda) = (f_1\pconv f_2)(x,1)}$.

  We now prove Part (i). Suppose that $x\in \argmin \penv$ and let
  $(\xbar,\bar \lambda) = \rpprox(x)$. Then $0\in \partial
  \penv(x)$. In view of \cref{th:polar-envelope-subdiff}(ii),
  there are two cases to consider, below.

  \paragraph{Case 1} $f^\pi(\xbar,\bar\lambda) < \alpha^{-1}
                \left\|
                    (x - \xbar,1 -\bar \lambda)
                \right\| $.
  This forces $\mu = 1$ in~\eqref{sub_formula}, which together with $0\in \partial \penv(x)$ and~\eqref{sub_formula2} implies the existence of $\beta$ such that
  \[
  (0,\beta) \in \alpha^{-1}\partial \|(x-\xbar,1-\bar \lambda)\|\cap \normal((\xbar,\bar \lambda)\mid\dom f^\pi).
  \]
  In particular, we have $x = \xbar$.
  Since $\alpha^{-1} \left\|(x-\xbar,1-\bar \lambda)\right\| > 0$, we must then have $\bar\lambda\neq 1$. Thus,
  \[
  \beta = \frac{1-\bar \lambda}{\alpha|1-\bar\lambda|}.
  \]
  Combining this with~\eqref{normalcone} and $(0,\beta)\in \normal((\xbar,\bar \lambda)\mid\dom f^\pi)$ yields
  \[
  0 \ge \ip{ (y,1)}{(0,1 - \bar \lambda)} = 1 - \bar \lambda\ \
  {\rm and}\ \ \ip{(\xbar,\bar \lambda)}{(0,1 - \bar \lambda)} = 0
  \]
  for any fixed $y\in \dom f$. The first relation above together with $\bar \lambda\neq 1$ yields $\bar \lambda > 1$, and hence we conclude that $\bar \lambda = 0$ by the second relation, which is a contradiction. Consequently, this case cannot happen.

  \paragraph{Case 2} $f^\pi(\xbar,\bar\lambda) = \alpha^{-1}
                \left\|
                    (x - \xbar,1 -\bar \lambda)
                \right\| $. In this case, we see from $0\in \partial \penv(x)$, \eqref{sub_formula2} and~\eqref{sub_formula} that there exist
$\beta$ and $\mu \in [0,1]$ that satisfy
\begin{equation}\label{importantrel}
(0,\beta) \in (\mu/\alpha) \partial \|(x - \xbar,1 -\bar \lambda)\|\cap \partial (\lscale{[1 - \mu]}f^\pi)(\xbar,\bar\lambda).
\end{equation}

We claim that $\mu > 0$. Assume to the contrary that $\mu = 0$. It
then follows from~\eqref{importantrel} that $\beta = 0$ and hence
  \[
  0 \in \partial f^\pi(\xbar,\bar\lambda),
  \]
meaning that $(\xbar,\bar\lambda)$ minimizes the gauge function $f^\pi$. Thus, $f^\pi(\xbar,\bar\lambda) = 0$. Because $f^\pi(\xbar,\bar\lambda) = \alpha^{-1}\left\| (x - \xbar,1 -\bar \lambda)\right\|$, we see further that $(\xbar,\bar\lambda) = (x,1)$. Hence,
  \[
  0 = f^\pi(\xbar,\bar\lambda) = f^\pi(x,1) = f(x).
  \]
Since the optimal value of~\eqref{eq:1} is positive, we have arrived at a contradiction. Consequently, we have shown that $\mu > 0$.

  From $\mu > 0$ and~\eqref{importantrel}, we see that $(0,\beta) \in (\mu/\alpha) \partial \|(x - \xbar,1 -\bar \lambda)\|$, which readily gives $x = \xbar$.

  Next, we claim that $\bar \lambda\neq 1$. Indeed, suppose to the contrary that $\bar \lambda=1$. Then
  \[
  f(\xbar) = f^\pi(\xbar,1)= f^\pi(\xbar,\bar\lambda) = \alpha^{-1}\left\|(x - \xbar,1 -\bar \lambda)\right\| = 0
  \]
  since $x = \xbar$. Because the optimal value of~\eqref{eq:1} is positive, this is a contradiction. Consequently, we must have $\bar \lambda\neq 1$.

  Now, we claim that $\mu < 1$. Suppose to the contrary that $\mu = 1$. Then we have from~\eqref{importantrel} and~\eqref{normalcone} that
  \begin{equation}\label{eq:13}
  \begin{aligned}
  (0,\beta)& \in \partial [\lscale{0}f^\pi](\xbar,\bar\lambda) = \normal((\xbar,\bar \lambda)\mid\dom f^\pi) \\
  &= \set{w \in [\dom f\times \{1\}]^-\mid \ip{w}{(\xbar,\bar \lambda)} = 0}.
  \end{aligned}
  \end{equation}
  In addition, we have from~\eqref{importantrel} that $(0,\beta) \in (1/\alpha) \partial \|(x - \xbar,1 -\bar \lambda)\|$, which together with $x = \xbar$ and $\bar \lambda\neq 1$ gives
  \begin{equation} \label{eq:14}
  \beta = \frac{1 - \bar \lambda}{\alpha|1 - \bar \lambda|}.
  \end{equation}
  Equations~\eqref{eq:13} and~\eqref{eq:14} together imply that
  \begin{equation*}
    0\ge \ip{(y,1)}{(0,1 - \bar \lambda)} = 1 - \bar \lambda
    \quad\mbox{and}\quad
    0 = \ip{(\xbar,\bar \lambda)}{(0,1 - \bar \lambda)}
  \end{equation*}
  for any $y\in \dom f$. As in Case~1, this yields a
  contradiction. Thus, we must also have $\mu < 1$.

  To summarize, we have shown that $\mu \in (0,1)$, $x = \xbar$, and
  $\bar \lambda\neq 1$. Together with~\eqref{importantrel}, we
  conclude that
  \begin{equation}\label{rel2}
    \frac{1}{1-\mu}(0,\beta)=
    \frac{\mu}{\alpha(1-\mu)}\left(0,\frac{1 - \bar \lambda}{|1 - \bar \lambda|}\right) \in \partial f^\pi(\xbar,\bar \lambda).
  \end{equation}
  We claim that $\bar \lambda > 0$. Suppose to the contrary that $\bar \lambda = 0$. Then due to $f^\pi(\xbar,\bar\lambda) = \alpha^{-1}
                \left\|
                    (x - \xbar,1 -\bar \lambda)
                \right\| $ and $x = \xbar$, we must have
  \[
  \rec{f}(\xbar) = f^\pi(\xbar,0) = f^\pi(\xbar,\bar\lambda) = \alpha^{-1} \left\| (x - \xbar,1 -\bar \lambda)\right\| = 1/\alpha > 0.
  \]
  However, from \cref{subdiff_pers}, \eqref{rel2}, and
  $\bar \lambda = 0$, we obtain
  \[
  0 \in \partial \rec{f}(\xbar).
  \]
  This implies that $\xbar$ minimizes the gauge function $\rec{f}$, meaning that $\rec{f}(\xbar) = 0$, which is a contradiction.

  The conclusion in Part~(i) now follows immediately from~\eqref{rel2}, the
  facts that $\bar \lambda > 0$ and $\xbar = x$, and
  \cref{subdiff_pers}.

  We now prove Part~(ii). Suppose that $x\in \argmin f$. Then
  $f(x) > 0$ because the optimal value of~\eqref{eq:1} is
  positive. Let $\lambda = [1+\alpha f(x)]^{-1}\in (0,1)$ and set
  $\mu = 1-\lambda$. Then
  \[
  f(x) = \frac{1-\lambda}{\alpha\lambda}.
  \]
  Since $\lambda > 0$, we have $0 \in \partial f(\lambda x/\lambda)$, which together with \cref{subdiff_pers} gives
  \begin{equation} \label{eq:15}
  (0,f(x)) = (0,-f^*(0))\in \partial f^\pi(\lambda x,\lambda),
  \end{equation}
  where the equality follows from a direct computation and the fact
  that $f(x) = \inf f$. Using this, $\lambda\in (0,1)$ and
  $f(x) = (1-\lambda)/(\alpha\lambda)=\mu/(\alpha(1-\mu))$,
  we see further that
  \begin{equation}\label{rel1}
  (0,\alpha^{-1}\mu)\in \frac{\mu}{\alpha}\partial\|(0,1-\lambda)\|\cap (1-\mu)\partial f^\pi(\lambda x,\lambda).
  \end{equation}
  Furthermore, we have
  \[
  f^\pi(\lambda x,\lambda) = \lambda f(x) = \frac{1-\lambda}{\alpha} = \alpha^{-1}\|(0,1-\lambda)\|,
  \]
  and from~\eqref{eq:15} and the fact that $\lambda\in(0,1)$,
  \[
  0 = (0,-\alpha^{-1}\mu + (1-\mu)f(x))\in \frac{\mu}{\alpha}\partial\|(0,\lambda - 1)\| + (1-\mu)\partial f^\pi(\lambda x,\lambda),
  \]
  showing that
  \[
  \begin{aligned}
  (1-\mu)f^\pi(\lambda x,\lambda) &+ \frac\mu\alpha\|(0,1-\lambda)\| \\
  &= \inf_{\xhat,\hat \lambda}\ \maxv{f^\pi(\xhat,\widehat \lambda),\,\alpha^{-1}\|(\xhat-\lambda x,\widehat \lambda-1)\|},
  \end{aligned}
  \]
  thanks to \citet[Corollary~2.8.15]{Za02}.
  This together with~\eqref{rel1} and \cref{Prop_sub} shows that
  \[
  (0,\alpha^{-1}\mu)\in \partial (f_1\pconv f_2)(\lambda x,1)
  \]
  with $f_1$ and $f_2$ defined in~\eqref{f1f2}. Invoking~\eqref{sub_formula2}, we conclude further that $0 \in \partial \penv(\lambda x)$, meaning that $\lambda x\in \argmin \penv$.
\end{proof}

Our next theorem states that one can obtain an optimal solution of~\eqref{eq:1} by considering some ``projected'' fixed points of the projected polar proximal map. This is an analogue of the well-known fact that for a proper closed convex function $h$, one has $\argmin h = \set{x\mid \prox_{\gamma h}(x) = x}$ for any $\gamma > 0$. %

\begin{theorem}[Projected fixed points of projected polar proximal map]\label{Thm2}
Consider $f$ as in~\eqref{eq:37} and let $\alpha > 0$.
\begin{enumerate}[{\rm (i)}]
  \item If $(x,\lambda) = \rpprox(x)$,
  then $\lambda > 0$ and $\lambda^{-1}x\in \argmin f$.
  \item If $x\in \argmin f$, then there exists $\lambda > 0$ so that
  $(\tau x,\lambda)= \rpprox(\tau x)$, where $\tau:=[1+\alpha f(x)]^{-1}$.
\end{enumerate}
\end{theorem}

\begin{proof}
  We first prove Part~(i). Suppose that $(x,\lambda) =
  \rpprox(x)$. According to
  \cref{th:polar-envelope-subdiff}(ii), there are two cases to
  consider.

  \paragraph{Case 1} $f^\pi(x,\lambda) < \alpha^{-1}\|(x,\lambda) - (x,1)\|$. In particular, $\lambda \neq 1$.
  Using this, \citet[Corollary~2.8.15]{Za02} and the fact that $(x,\lambda)$ is a minimizer, we see that
  \[
  \begin{aligned}
  0 &\in \alpha^{-1}\partial \|\cdot - (x, 1)\|(x,\lambda) + \normal((x,\lambda)\mid\dom f^\pi)\\
  &= \frac1{\alpha|1 - \lambda|}(0,\lambda - 1) + \normal((x,\lambda)\mid\dom f^\pi).
  \end{aligned}
\]
We use~\eqref{normalcone} to obtain the equivalent expression
\[
  \frac1{\alpha|1 - \lambda|}(0,1 - \lambda)\in \set{w \in [\dom f\times \{1\}]^-\mid \ip{ w}{(x,\lambda)} = 0}.
  \]
   Thus, for any $y\in \dom f$, we have
  \[
    0 \ge \ip{(y,1)}{(0,1 - \lambda)}
    \quad\mbox{and}\quad
    0 = \ip{(x,\lambda)}{(0,1 - \lambda)}.
  \]
  Since $\lambda \neq 1$, the first inequality above gives $\lambda > 1$, which, together with the second relation above gives $\lambda = 0$, leading to a contradiction. Thus, this case cannot happen.

  \paragraph{Case 2} $f^\pi(x,\lambda) = \alpha^{-1}\|(x,\lambda) - (x,1)\|$. We first claim that $\lambda \neq 1$. Suppose to the contrary that $\lambda = 1$. Then $f^\pi(x,1) = \alpha^{-1}\|(x,1) - (x,1)\| = 0$. From this and the definition of $f^\pi$, we then have $f(x) = 0$, which is impossible because the optimal value of~\eqref{eq:1} is positive. Thus, $\lambda \neq 1$.

  Now, using this, \citet[Corollary~2.8.15]{Za02}, and the fact that
  $(x,\lambda)$ is a minimizer, we have
  \[
  0 = u + v
  \]
  for some $u \in \partial (\lscale{\mu} f^\pi)(x,\lambda)$,
  $v=\frac{1-\mu}{\alpha|1-\lambda|}(0,\lambda - 1)$, and
  $\mu \in [0,1]$.

  We claim that $\mu > 0$. Suppose to the contrary that $\mu = 0$. Then we have
  \[
    0 \in \frac{1}{\alpha|1-\lambda|}(0,\lambda - 1) + \normal((x,\lambda)\mid\dom f^\pi).
  \]
  A contradiction can be derived exactly as in Case 1. Thus, $\mu > 0$.

  Next, we claim that $\lambda > 0$. Suppose to the contrary that
  $\lambda = 0$. Because $\mu > 0$,
  \[
  \mu^{-1}u\in \partial f^\pi(x,0).
  \]
  On the other hand, we also have from $u + v = 0$ that
  \[
  u = - v = \frac{1-\mu}{\alpha|1-\lambda|}(0,1 - \lambda) = \frac{1-\mu}{\alpha}(0,1).
  \]
  Combining the previous two displays with
  \cref{subdiff_pers}, we conclude that
  $0 \in \partial \rec{f}(x)$. This implies that $x$ minimizes the
  gauge function $\rec{f}$, meaning that
  \[
  0 = \rec{f}(x) = f^\pi(x,0) = \alpha^{-1}\|(x,0) - (x,1)\| = \alpha^{-1} > 0,
  \]
  which is a contradiction. Thus, $\lambda > 0$.

  Now, with $\mu > 0$, we have from
  $u \in \partial (\lscale{\mu} f^\pi)(x,\lambda)$ and $u + v = 0$
  that
  \[
    \frac{u}{\mu}\in \partial f^\pi(x,\lambda)
    \quad\mbox{and}\quad
    u = - v = \frac{1-\mu}{\alpha|1-\lambda|}(0,1 - \lambda).
  \]
  Using this last display, $\lambda > 0$, and
  \cref{subdiff_pers}, we conclude that
  $0 \in \partial f(\lambda^{-1}x)$, as desired. This proves Part~(i).

  We now prove Part~(ii). Suppose that $x\in \argmin f$. By
  \cref{Thm1}(ii), we have $\tau x\in \argmin \penv$, where
  $\tau = [1+\alpha f(x)]^{-1}$. Now let
  $(\tilde x,\lambda)= \rpprox(\tau x)$. Then \cref{Thm1}(i)
  gives $\tilde x = \tau x$ and $\lambda > 0$.
\end{proof}

\subsection{Projected polar proximal-point algorithm and its variants}

We are motivated by the optimality conditions in \cref{Thm2} to
consider a fixed-point iteration for solving~\eqref{eq:1}. This
iteration mirrors the well-known proximal-point algorithm
\citep{martinet:1970}, which is a fixed-point iteration based on the
(usual) proximal map. We call our algorithm the {\em projected polar
  proximal-point algorithm}.

\begin{mdframed}
{\bf Projected polar proximal-point algorithm (P$^4$A)}: Fix any $\alpha > 0$, start with any $x_0$ and update
\begin{equation*}
(x_{k+1},\lambda_{k+1}) = \rpprox(x_k).
\end{equation*}
\end{mdframed}

Let $\set{(x_k,\lambda_k)}$ be a sequence generated by {\bf
  P$^4$A}. If $\|x_{k+1}-x_k\|\to 0$ and if $(x_*,\lambda_*)$ is an
accumulation point of $\{(x_k,\lambda_k)\}$, then it is routine to
show that
\begin{equation*}%
(x_*,\lambda_*) = \rpprox(x_*),
\end{equation*}
which according to \cref{Thm2} implies that $\lambda_* > 0$ and
$\lambda_*^{-1}x_*\in \argmin f$. In the next theorem, we will show
that $\|x_{k+1}-x_k\|\to 0$ under an additional assumption.

\begin{theorem}[Convergence of {\bf P$^4$A}]\label{Thm:P4A}
  Consider $f$ as in~\eqref{eq:37} and let $\alpha > 0$. Let $\{(x_k,\lambda_k)\}$ be generated by {\bf P$^4$A}. Then $\penv(x_{k+1})\le \penv(x_{k})$ for all $k$. If in addition there exists $\gamma > 0$ so that the function $(x,\lambda)\mapsto [f^\pi(x,\lambda)]^2 - (\gamma/2)\|x\|^2$ is convex, then $\|x_{k+1}-x_k\|\to 0$.
\end{theorem}

\begin{remark}[Comments on the convexity condition]
  We give a simple sufficient condition for the convexity condition
  used in the hypothesis of \cref{Thm:P4A}. Let $\Cscr$ be a closed
  convex set that does not contain the origin and $\rho$ be a closed
  gauge. For any fixed $\epsilon > 0$, consider the function
  $\kappa(x) := \sqrt{\rho^2(x) + \epsilon\|x\|^2}$, which is a
  perturbation of the gauge $\rho$. Then $\kappa$ is a gauge function
  \citep[Corollary~15.3.1]{roc70} and $\kappa^2$ is strongly convex,
  say, with modulus $\gamma > 0$. If we set
  $f(x) := \kappa(x) + \delta\Cs(x)$, then
  \[
  f^\pi(x,\lambda) = \begin{cases}
    \kappa(x) + \delta_{\lambda\Cscr}(x) &\mbox{if $\lambda > 0$,}\\
    \kappa(x) + \delta_{\Cscr_\infty}(x) & \mbox{if $\lambda = 0$.}
  \end{cases}
  \]
  It is routine to show that $(x,\lambda)\mapsto [f^\pi(x,\lambda)]^2 - (\gamma/2)\|x\|^2$ is convex.
\end{remark}

\begin{proof}
  Let $(x_{k+2},\lambda_{k+2})=\rpprox(x_{k+1})$ and $(x_{k+1},\lambda_{k+1})=\rpprox(x_k)$. Then,
  \begin{equation}\label{weirdrel}
  \begin{aligned}
    &\penv(x_{k+1})  \overset{\rm(a)}= \max\set{f^\pi(x_{k+2},\lambda_{k+2}),\,\alpha^{-1}\|(x_{k+2}-x_{k+1},\lambda_{k+2}-1)\|} \\
    &\overset{\rm(b)}\le \max\big\{f^\pi(x_{k+1},\lambda_{k+1}),\,\alpha^{-1}\|(x_{k+1}-x_{k+1},\lambda_{k+1}-1)\|\big\} \\
    &= \max\big\{f^\pi(x_{k+1},\lambda_{k+1}),\,\alpha^{-1}|\lambda_{k+1}-1|\big\}\\
    &\le \max\big\{f^\pi(x_{k+1},\lambda_{k+1}),\,\alpha^{-1}\|(x_{k+1}-x_{k},\lambda_{k+1}-1)\|\big\}
    \overset{\rm(c)}= \penv(x_{k}),
  \end{aligned}
  \end{equation}
  where (a) and (c) follow from the definition of $\penv$, and (b)
  follows from the definition of $(x_{k+2},\lambda_{k+2})$ as a
  minimizer, so that the function value at $(x_{k+2},\lambda_{k+2})$
  is less than that at $(x_{k+1},\lambda_{k+1})$. This proves that
  $\penv(x_{k+1})\le\penv(x_k)$, as required.

  Now suppose in addition that
  $(x,\lambda)\mapsto [f^\pi(x,\lambda)]^2 - (\gamma/2)\|x\|^2$ is
  convex. Let $\tau = \min\set{\gamma,\alpha^{-2}}$. Then for each
  $k$, the function
  \[
  [f^\pi(x,\lambda)]^2 - (\tau/2)\|x - x_k\|^2
  \]
  is convex. Thus, the function
  \[
  G_k(x,\lambda) := \max\left\{[f^\pi(x,\lambda)]^2,\alpha^{-2}\|(x-x_{k+1},\lambda-1)\|^2\right\}- (\tau/2)\|x - x_{k+2}\|^2
  \]
  is convex. Since $(x_{k+2},\lambda_{k+2})$ minimizes the convex function
  \[
  (x,\lambda)\mapsto \max\big\{[f^\pi(x,\lambda)]^2,\alpha^{-2}\|(x-x_{k+1},\lambda-1)\|^2\big\},
  \]
  the first-order optimality condition implies that
  $(x_{k+2},\lambda_{k+2})$ also minimizes $G_k$. Thus,
  \[
  \begin{aligned}
   &\penv^2(x_{k+1}) = \max\set{[f^\pi(x_{k+2},\lambda_{k+2})]^2,\alpha^{-2}\|(x_{k+2}-x_{k+1},\lambda_{k+2}-1)\|^2}\\
   & = G_k(x_{k+2},\lambda_{k+2})\le G_k(x_{k+1},\lambda_{k+1})\\
   &= \max\big\{[f^\pi(x_{k+1},\lambda_{k+1})]^2,\alpha^{-2}\|(x_{k+1}-x_{k+1},\lambda_{k+1}-1)\|^2\big\} - (\tau/2)\|x_{k+2} - x_{k+1}\|^2\\
   & \le \penv^2(x_{k})- (\tau/2)\|x_{k+2} - x_{k+1}\|^2,
  \end{aligned}
  \]
  where the last inequality follows from~\eqref{weirdrel}. Rearranging terms in the above inequality and summing from $k=0$ to $\infty$, we obtain
  \[
  \begin{aligned}
  \sum_{k=0}^\infty \frac\tau2\|x_{k+2} - x_{k+1}\|^2 &\le \sum_{k=0}^\infty(\penv^2(x_{k}) - \penv^2(x_{k+1}))\\
  &\le \penv^2(x_0) < \infty,
  \end{aligned}
  \]
  which implies that $\|x_{k+1}-x_k\|\to 0$, as desired.
\end{proof}

We next describe another natural algorithm for solving~\eqref{eq:1},
motivated by \cref{Thm1} instead. This theorem asserts that we
can recover a solution of~\eqref{eq:1} by computing $\rpprox$ at a
minimizer of $\penv$. Thus, in order to solve~\eqref{eq:1}, one can
just minimize the projected polar envelope $\penv$. Specifically, we
can apply a variant of the steepest-descent algorithm with line search
to minimize $\penv$. Such an approach can be understood as a
proximal-point-like algorithm in the same manner that the classical
proximal-point algorithm can be regarded as a steepest-descent
algorithm applied to the Moreau envelope.

\begin{mdframed}
{\bf Projected polar envelope minimization algorithm (EMA)}: Start with
any $x_0$ and $\sigma\in (0,1)$. For each $k$, pick
$\beta_k > 0$ (so that $0<\inf \beta_k \le \sup \beta_k < \infty$) and perform the following iteration:
\begin{enumerate}
  \item Find the smallest integer $t\ge 0$ so that
    \[\penv(x_k - 2^{-t} \beta_k\nabla \penv(x_k)) \le \penv(x_k) - \sigma 2^{-t} \beta_k\|\nabla \penv(x_k)\|^2.\]
  \item Set $x_{k+1} = x_k - 2^{-t}\beta_k\nabla \penv(x_k)$.
\end{enumerate}
\end{mdframed}

We have the following convergence result concerning {\bf EMA}.

\begin{theorem}[Convergence of {\bf EMA}]
  Consider $f$ as in~\eqref{eq:37} and let $\alpha > 0$. Let $\set{x_k}$ be generated by {\bf EMA}. Then $\set{x_k}$ converges to a global minimizer of $\penv$.
\end{theorem}
\begin{proof}
  We first claim that $\penv(x) > 0$ for all $x$. Suppose to the contrary that $\penv(x) = 0$ for some $x$. Then we must have $f^\pi(\bar x,\bar \lambda) = (1/\alpha)\|(\bar x,\bar \lambda) - (x,1)\| = 0$, where $(\bar x,\bar \lambda) \in \rpprox(x)$. Thus, $\bar x = x$ and $\bar \lambda = 1$. Consequently, we deduce that
  \[
  0 = f^\pi(\bar x,\bar \lambda) = f^\pi(x,1) = f(x),
  \]
  contradicting the assumption that~\eqref{eq:1} has a positive
  optimal value. Thus, $\penv(x) > 0$ for all $x$. Together with
  \cref{th:polar-envelope-subdiff}(iii) and (iv), we deduce
  further that $\penv$ is a gauge function with a continuous
  gradient. Moreover, using \cref{Thm1} together with the
  assumption that $\argmin f\neq \emptyset$ (see~\eqref{eq:37}), we have
  $\argmin \penv \neq \emptyset$. The desired conclusion now follows
  from these and \citet[Theorem~1]{Iusem:2003}.
\end{proof}

{\bf EMA} is closely related to {\bf P$^4$A}. Indeed, according to
\cref{th:polar-envelope-subdiff}(iv),
$\nabla \penv(x)$ is a positive scaling of
$x - \xbar$, where
$(\xbar, \bar \lambda) =
\rpprox(x)$. Thus, one can think of {\bf EMA} as a version of {\bf P$^4$A} with a
line-search strategy incorporated to guarantee
convergence.

\section{Concluding remarks}

This paper continues the authors' investigation into the optimization
of gauge functions and their applications
\citep{FriedlanderMacedoPong:2014,ABDFM18}. Gauge functions seem to
appear naturally when modelling certain classes of inverse problems,
including sparse optimization and machine learning. Our goal is to
help establish the foundation for tools and algorithms specialized to
this family of problems.

Several research avenues arise from our results and remain to be
explored. The polar envelope and proximal map share many of the
features of the Moreau envelope and its proximal map, and we have
highlighted many of these. However, we have yet to discover an
analogue for the Moreau identity, which relates the proximal map of a
convex function with the map of its Fenchel conjugate. Can we
similarly connect the polar proximal map of a gauge function with the
map of its polar? Only the simple case is obvious: when the functions
are indicators to a closed convex cone, the (Moreau) proximal and
polar proximal maps coincide (cf.~\cref{ex:indicator-to-cone}).

Along similar lines, the referees raised several important open
questions that we are not yet able to answer. Does the duality between
strong convexity and Lipschitz differentiability of the conjugate
operation have an analogue under the polarity operation?  Is it
possible to specify conditions under which the regularized
problem~\eqref{eq:bp} exhibits an exactness property and recovers a
solution of the unregularized problem?

The polar proximal-point algorithms that we proposed in
\cref{sec:polar-prox-algo} represent a first step in developing
practical algorithms, in the same way that, say, the proximal-point
algorithm can be used as a springboard for many other algorithms used
in practice, such as augmented Lagrangian, proximal-bundle, and
proximal-gradient algorithms. What are the implementable forms of the
polar proximal-point algorithms?

\section*{Acknowledgments}

The authors extend thanks to Alberto Seeger, who patiently answered
our queries regarding his work on the max convolution operation that
laid the foundation for our investigation, and to Constantin
Z\u{a}linescu for providing us with useful references. We are also
grateful to two anonymous referees for insightful comments and
supplying future avenues of research.

  \bibliographystyle{abbrvnat}
  \bibliography{shorttitles,master,friedlander}

\end{document}
